\documentclass[a4paper,10pt]{article}

\usepackage{microtype}

\usepackage{setspace}

\AtBeginDocument{%
	\setlength{\abovedisplayskip}{7pt plus 2pt minus 2pt}%
	\setlength{\belowdisplayskip}{7pt plus 2pt minus 2pt}%
	\setlength{\abovedisplayshortskip}{5pt plus 2pt}%
	\setlength{\belowdisplayshortskip}{5pt plus 2pt minus 2pt}%
}

\usepackage[
a4paper,			
top=27mm,			
bottom=29mm,		
left=30mm,			
right=30mm,			
headheight=14pt,	%
headsep=8mm,		%
footskip=12mm		%
]{geometry}

\usepackage[T1]{fontenc}

\usepackage[UTF8,scheme=plain,fontset=none]{ctex}

\usepackage{amsmath,mathtools}

\usepackage{amssymb}

\usepackage{amsthm}

\allowdisplaybreaks[4]

\usepackage{graphicx}

\usepackage{subcaption}

\usepackage{booktabs}

\usepackage{array}

\usepackage{enumitem}

\usepackage{xcolor}

\usepackage{titlesec}

\usepackage{fancyhdr}

\usepackage{etoolbox}

\usepackage{xparse}

\usepackage{algorithm}

\usepackage{algorithmicx}

\usepackage{aliascnt} 

\definecolor{PaperLinkBlue}{HTML}{1F4E79}

\usepackage[
unicode=true,			
colorlinks=true,		
linkcolor=PaperLinkBlue,
citecolor=PaperLinkBlue,
urlcolor=PaperLinkBlue,	
pdfborder={0 0 0}		
]{hyperref}

\usepackage[
nameinlink,		
noabbrev,		
capitalize		
]{cleveref}

\usepackage[numbers,sort&compress]{natbib}  

\titleformat{\section}
{\large\bfseries}		
{\thesection.}			
{0.65em}				
{}

\titleformat{\subsection}
{\normalsize\bfseries}
{\thesubsection.}
{0.65em}
{}

\titleformat{\subsubsection}
{\normalsize\itshape}
{\thesubsubsection.}
{0.65em}
{}

\titlespacing*{\section}
{0pt}
{2.2ex plus 0.5ex minus 0.2ex}
{0.9ex}

\titlespacing*{\subsection}
{0pt}
{1.8ex plus 0.4ex minus 0.2ex}
{0.6ex}

\titlespacing*{\subsubsection}
{0pt}
{1.5ex plus 0.3ex minus 0.2ex}
{0.4ex}

\numberwithin{equation}{section}

\theoremstyle{plain}

\newtheorem{theorem}{Theorem}[section]

\newaliascnt{lemma}{theorem}
\newtheorem{lemma}[lemma]{Lemma}
\aliascntresetthe{lemma}

\newaliascnt{proposition}{theorem}
\newtheorem{proposition}[proposition]{Proposition}
\aliascntresetthe{proposition}

\newaliascnt{corollary}{theorem}

\aliascntresetthe{corollary}

\theoremstyle{definition}

\newaliascnt{definition}{theorem}
\newtheorem{definition}[definition]{Definition}
\aliascntresetthe{definition}

\newaliascnt{assumption}{theorem}
\newtheorem{assumption}[assumption]{Assumption}
\aliascntresetthe{assumption}

\newaliascnt{example}{theorem}
\newtheorem{example}[example]{Example}
\aliascntresetthe{example}

\theoremstyle{remark}

\newaliascnt{remark}{theorem}
\newtheorem{remark}[remark]{Remark}
\aliascntresetthe{remark}

\crefname{theorem}{Theorem}{Theorems}
\crefname{lemma}{Lemma}{Lemmas}
\crefname{proposition}{Proposition}{Propositions}
\crefname{corollary}{Corollary}{Corollaries}
\crefname{definition}{Definition}{Definitions}
\crefname{assumption}{Assumption}{Assumptions}
\crefname{example}{Example}{Examples}
\crefname{remark}{Remark}{Remarks}
\crefname{equation}{Equation}{Equations}
\crefname{section}{Section}{Sections}
\crefname{algorithm}{Algorithm}{Algorithms}

\setlist{
	topsep=0.5em,		
	itemsep=0.2em,		
	parsep=0pt,			
	partopsep=0pt,		
}

\setlist[itemize,1]{
	label=\textbullet,
	leftmargin=2em,
	labelsep=0.6em
}

\setlist[itemize,2]{
	label=\textendash,
	leftmargin=2em,
	labelsep=0.6em
}

\setlist[itemize,3]{
	label=\(\circ\),
	leftmargin=2em,
	labelsep=0.6em
}

\setlist[enumerate,1]{
	label={\upshape(\roman*)},
	ref=\roman*,
	leftmargin=2.6em,
	labelsep=0.6em
}

\setlist[enumerate,2]{
	label=(\alph*),
	ref=\alph*,
	leftmargin=2.4em,
	labelsep=0.6em
}

\newcommand{\RR}{\mathbb{R}}
\newcommand{\NN}{\mathbb{N}}

\newcommand{\QQ}{\mathbb{Q}}

\renewcommand{\SS}{\mathbb{S}}
\newcommand{\LT}{L^2}
\newcommand{\LTC}{L^2(0,T;\RR^m)}
\newcommand{\LB}{L^\beta}

\newcommand{\LBC}{L^\beta(0,T;\RR^m)}
\newcommand{\LI}{L^\infty}
\newcommand{\LIC}{L^\infty(0,T;\RR^m)}

\newcommand{\eps}{\varepsilon}

\newcommand{\dd}{\mathop{}\!\mathrm{d}}

\DeclarePairedDelimiter{\norm}{\lVert}{\rVert}

\DeclarePairedDelimiter{\abs}{\lvert}{\rvert}

\DeclarePairedDelimiterX{\inner}[2]
{\langle}
{\rangle}
{#1,#2}

\DeclareMathOperator*{\esssup}{ess\,sup}

\providecommand{\PaperTitleText}{}
\providecommand{\ShortTitleText}{}
\providecommand{\ShortAuthorsText}{}
\providecommand{\PaperAuthorsText}{}
\providecommand{\PaperAffiliationsText}{}
\providecommand{\PaperAbstractText}{}
\providecommand{\PaperKeywordsText}{}
\providecommand{\PaperMSCText}{}
\providecommand{\PaperPDFSubjectText}
{Mathematical research article}

\NewDocumentCommand{\PaperTitle}{+m}{%
	\renewcommand{\PaperTitleText}{#1}%
}

\NewDocumentCommand{\ShortTitle}{+m}{%
	\renewcommand{\ShortTitleText}{#1}%
}

\NewDocumentCommand{\ShortAuthors}{+m}{%
	\renewcommand{\ShortAuthorsText}{#1}%
}

\NewDocumentCommand{\PaperAuthors}{+m}{%
	\renewcommand{\PaperAuthorsText}{#1}%
}

\NewDocumentCommand{\PaperAffiliations}{+m}{%
	\renewcommand{\PaperAffiliationsText}{#1}%
}

\NewDocumentCommand{\PaperAbstract}{+m}{%
	\renewcommand{\PaperAbstractText}{#1}%
}

\NewDocumentCommand{\PaperKeywords}{+m}{%
	\renewcommand{\PaperKeywordsText}{#1}%
}

\NewDocumentCommand{\PaperMSC}{+m}{%
	\renewcommand{\PaperMSCText}{#1}%
}

\providecommand{\PaperAbstractLabelText}{Abstract}

\providecommand{\PaperKeywordsLabelText}{Keywords}

\providecommand{\PaperMSCLabelText}
{Mathematics Subject Classification (2020)}

\newcommand{\AffMark}[1]{%
	\textsuperscript{%
		\hyperlink{aff:#1}{#1}%
	}%
}

\newcommand{\CorMark}{%
	\textsuperscript{%
		,\hyperlink{corresponding-author}{*}%
	}%
}

\newcommand{\PaperTemplateWarnIfEmpty}[2]{%
	\ifdefempty{#1}{%
		\PackageWarning
		{paper-template}
		{Paper field `#2' is empty}%
	}{}%
}

\AtBeginDocument{%
	\PaperTemplateWarnIfEmpty
	{\PaperTitleText}
	{PaperTitle}%
	
	\PaperTemplateWarnIfEmpty
	{\ShortTitleText}
	{ShortTitle}%
	
	\PaperTemplateWarnIfEmpty
	{\ShortAuthorsText}
	{ShortAuthors}%
	
	\PaperTemplateWarnIfEmpty
	{\PaperAuthorsText}
	{PaperAuthors}%
	
	\PaperTemplateWarnIfEmpty
	{\PaperAffiliationsText}
	{PaperAffiliations}%
	
	\PaperTemplateWarnIfEmpty
	{\PaperAbstractText}
	{PaperAbstract}%
	
	\PaperTemplateWarnIfEmpty
	{\PaperKeywordsText}
	{PaperKeywords}%
	
	\PaperTemplateWarnIfEmpty
	{\PaperMSCText}
	{PaperMSC}%
}

	\NewDocumentCommand{\MakePaperTitle}{}{%
		\thispagestyle{paperfirstpage}%
		
		\begin{center}
			
			{\fontsize{17}{21}\selectfont
				\bfseries
				\PaperTitleText
				\par
			}
			
			\vspace{1.2em}
			
			{\large
				\PaperAuthorsText
				\par
			}
			
			\vspace{0.85em}
			
			\begin{minipage}{0.92\textwidth}
				\centering
				\small
				\PaperAffiliationsText
			\end{minipage}
			
		\end{center}
		
		\vspace{1.25em}
		
		\begin{center}
			
			\begin{minipage}{0.92\textwidth}
				\small
				
				\noindent
				\textbf{\PaperAbstractLabelText.}
				\enspace
				\PaperAbstractText
				\par
				
				\vspace{0.65em}
				
				\noindent
				\textbf{\PaperKeywordsLabelText.}
				\enspace
				\PaperKeywordsText
				\par
				
				\vspace{0.35em}
				
				\noindent
				\textbf{\PaperMSCLabelText.}
				\enspace
				\PaperMSCText
				
			\end{minipage}
			
		\end{center}
		
		\vspace{1.2em}
	}

	\fancypagestyle{paperbody}{%
		\fancyhf{}%
		
		\fancyhead[L]{%
			\small
			\itshape
			\ShortTitleText
		}%
		
		\fancyhead[R]{%
			\small
			\ShortAuthorsText
		}%
		
		\fancyfoot[C]{%
			\thepage
		}%
		
	}

	\fancypagestyle{paperfirstpage}{%
		\fancyhf{}%
		
		\fancyfoot[C]{%
			\thepage
		}%
		
	}
	
		\AtBeginDocument{%
			\hypersetup{
				pdftitle={\PaperTitleText},
				pdfauthor={\ShortAuthorsText},
				pdfsubject={\PaperPDFSubjectText},
				pdfkeywords={\PaperKeywordsText}
			}%
		}

\PaperTitle{Convergence Analysis of Newton Methods for Nonlinear Optimal Control Problems}

\ShortTitle{}

\ShortAuthors{}

\PaperAuthors{%
	Yushu Feng\AffMark{1}
	\quad and \quad
	Haisen Zhang\AffMark{2}\CorMark
}

\PaperAffiliations{%
	\hypertarget{aff:1}{}\textsuperscript{1}\,
	School of Mathematics, Sichuan University, Chengdu, China.\\[0.25em]
	\hypertarget{aff:2}{}\textsuperscript{2}\,
	School of Mathematical Sciences, Sichuan Normal University, Chengdu, China.\\[0.25em]
	\hypertarget{corresponding-author}{}\textsuperscript{*}\,
	Corresponding author:
	\href{mailto:haisenzhang@yeah.net}{\texttt{haisenzhang@yeah.net}}.
}

\PaperAbstract{%
	The main purpose of this paper is to establish the local quadratic convergence of Newton's method for nonlinear optimal control problems.
	Under suitable smoothness assumptions, the cost functional is first shown to be twice G{\^a}teaux differentiable with respect to the control. An explicit operator representation of its second derivative is derived, and its continuity in different control spaces is examined. The second-order coercivity condition in $\LTC$ is then shown to be equivalent to the existence of a strongly regular solution to an associated matrix Riccati differential equation, thereby expressing an infinite-dimensional quadratic-form condition in terms of the solvability of a matrix differential equation. On this basis, local quadratic convergence of Newton's method in $\LIC$ is established using linear-quadratic optimal control theory. Under suitable structural assumptions, local quadratic convergence in $\LTC$ is also established. Explicit estimates of the corresponding convergence neighborhoods are derived, and the theoretical results are illustrated by numerical examples.
}

\PaperKeywords{%
	nonlinear optimal control problems, Newton's method, second-order coercivity conditions, Riccati equation, local quadratic convergence
}

\PaperMSC{49M15, 49K15, 65K10}

\begin{document}

\MakePaperTitle


\section{Introduction}

Given $T>0$, consider the nonlinear control system
\begin{equation}\label{system}
	\left\{
	\begin{aligned}
		& \dot{x}(t) = b(t,x(t),u(t)), \quad t \in [0,T], \\
		& x(0) = x_0,
	\end{aligned}
	\right.
\end{equation}
and the cost functional
\begin{equation}\label{cost}
	J(u(\cdot)) = \int_{0}^{T}f(t,x(t),u(t)) \dd t + g(x(T)).
\end{equation}
Here, $u: [0,T] \to \mathbb{R}^m$ is the control and $x: [0,T] \to \mathbb{R}^n$ is the corresponding state. We consider the following deterministic optimal control problem: find $\bar{u(\cdot)} \in \LTC$ such that
\begin{equation}\label{DOCP}
	J(\bar{u(\cdot)})
	= \underset{u(\cdot)\in \LTC}{\inf}J(u(\cdot)).
\end{equation}

Newton-type methods are widely used in nonlinear optimization and optimal control, largely because of their local quadratic convergence. For problem \eqref{DOCP}, two common formulations lead to the following methods.
\begin{enumerate}
	\item Newton's method: problem \eqref{DOCP} is viewed as an unconstrained optimization problem on the infinite-dimensional function space $\LTC$, and the cost functional is approximated directly to second order. At each iteration, the Newton direction formally satisfies
	\begin{equation*}
		J''(u^k(\cdot)) v^k(\cdot) = -J'(u^k(\cdot)),
	\end{equation*}
	and it can be obtained by solving a linear-quadratic optimal control problem. The iteration is specified in \cref{Algorithm-DOCP-Newton}.
	
	\item Lagrange--Newton and SQP methods: problem \eqref{DOCP} is viewed as a structured constrained optimization problem in which the state and control are optimization variables and the state equation is an equality constraint. Introduce the adjoint equation
	\begin{equation}\label{adjoint}
		\left\{
		\begin{aligned}
			& \dot{p}(t)
			= -b_x(t,x(t),u(t))^\top p(t)-f_x(t,x(t),u(t)), \quad t\in[0,T], \\
			& p(T) = g_x(x(T)),
		\end{aligned}
		\right.
	\end{equation}
	and the Hamiltonian
	\begin{equation} \label{Hamiltonian}
		H(t,x,u,p) = b(t,x,u)^\top p + f(t,x,u).
	\end{equation}
	The stationarity condition is
	\begin{equation}\label{FNC}
		H_u(t,\bar{x}(t),\bar{u}(t),\bar{p}(t)) = 0,
		\quad \text{for a.e. }t \in [0,T],
	\end{equation}
	where $\bar{u(\cdot)}\in\LTC$, and $\bar{x}(\cdot),\bar{p}(\cdot)$ denote the corresponding state and adjoint, respectively. The Lagrange--Newton method applies Newton's method directly to the optimality system consisting of the state equation \eqref{system}, the adjoint equation \eqref{adjoint}, and the stationarity condition \eqref{FNC} \cite{Tapia1974,AtkinsonHan2005}. Each iteration therefore requires the solution of a linearized optimality system. Sequential quadratic programming (SQP), in turn, linearizes the state equation \eqref{system} and constructs a quadratic subproblem by taking a second-order approximation of the Lagrangian. Each SQP iteration also requires the solution of a linear-quadratic optimal control problem; see \cref{RSQP} for the precise formulation. Under suitable conditions, the first-order optimality system of the SQP subproblem coincides with the linearized optimality system of the Lagrange--Newton method. Thus, despite their different formulations, the two methods are equivalent. We refer to both as SQP methods below. Further details can be found in \cite{Tapia1978,Powell1978,Alt1990,Powell1978MP,PalonmaresMagasatian1976,Han1976} and the survey \cite{Boggs1995}.
\end{enumerate}

The local quadratic convergence of Newton's method in an infinite-dimensional space is usually established under two conditions: bounded invertibility of the second derivative at the solution and local Lipschitz continuity of the second derivative in the same space. The former often follows from a second-order coercivity condition. In optimal control, the integral quadratic structure of the second variation makes $\LTC$ the natural space for the second-order coercivity condition; see \cref{LSSC}. On the other hand, $J$ is typically twice Fr{\'e}chet differentiable with respect to $u(\cdot)$ in $\LIC$, where its second derivative is locally Lipschitz continuous. This difference between the norms used for smoothness and coercivity is known as the two-norm discrepancy; see \cite{Maurer1981,AltMalanowski1993,HoppeNeitzel2021,Troltzsch2010}. A similar difficulty arises in the convergence analysis of SQP methods for optimal control: second-order coercivity is imposed with respect to a weaker norm on the state--control space, whereas second-order smoothness is established with respect to a stronger norm.

\subsection{Related work}

Early applications of Newton-type methods to optimal control focused primarily on the construction and numerical implementation of the algorithms, including Newton's method \cite{KellyKoppMoyer1963,Mitter1966}, the Lagrange--Newton method \cite{KoppMcGill1964,McGill1965,SchleyLee1967}, and SQP methods \cite{Machielsen1988}. Subsequent work developed a systematic theory of local quadratic convergence for SQP methods based on second-order smoothness and stability of the subproblems. Two main approaches address the two-norm discrepancy. In the first approach, coercivity is imposed in the weaker norm, whereas unique solvability and stability of the subproblems are established in the stronger space. Convergence is then obtained from smoothness in the stronger norm. The work of Maurer \cite{Maurer1981} and Malanowski \cite{Malanowski1993} laid the foundations for this approach. In the second approach, additional structural assumptions allow the required smoothness to be established in the weaker norm, in which coercivity is already available. Alt \cite{Alt1990} studied local quadratic convergence of infinite-dimensional SQP methods and applied the results to optimal control problems with dynamics affine in the control and a cost functional quadratic in the control, obtaining convergence in the weaker norm. For more general problems, Alt and Malanowski \cite{AltMalanowski1993} and Dontchev et al.\ \cite{DontchevHagerPoore1995} established local quadratic convergence of SQP methods in the stronger norm, and Alt \cite{Alt1994} also derived estimates of the convergence radius. These analyses were subsequently extended to PDE-constrained optimal control, taking into account the structure and regularity of the governing equations. Early examples include semilinear elliptic boundary control \cite{Heinkenschloss1996}, phase-field control \cite{HeinkenschlossTroltzsch1998}, and semilinear parabolic control \cite{Troltzsch1999}.

For a class of weakly singular optimal control problems, Ito and Kunisch \cite{ItoKunisch2000} used regularity of the state and adjoint equations, together with second-order smoothness of the cost functional $J(x(\cdot),u(\cdot))$ in the weaker state--control space $X_1 \times U$ and uniform invertibility of the linearized optimality system in the weaker state--control--adjoint space $X_1 \times U \times Z_1$. Under these and related assumptions, they proved local quadratic convergence of the controls in $U$ and of the state--control--adjoint triples in the stronger space $X_2 \times U \times Z_2$. Hinze and Kunisch \cite{HinzeKunisch2001} considered optimal control of the two-dimensional unsteady Navier--Stokes equations with the control entering linearly. For a separable cost functional $J(x(\cdot),u(\cdot)) = J_1(x(\cdot)) + J_2(u(\cdot))$ satisfying second-order smoothness assumptions in the joint state--control variable, they used the smoothness of the control-to-state map to prove local quadratic convergence of Newton's method and SQP methods, under assumptions including positive semidefiniteness of $J_{xx}$ and coercivity of $J_{uu}$ at the solution.

Further results on Newton-type methods for optimal control have been obtained in recent years. Hoppe and Neitzel \cite{HoppeNeitzel2021} established convergence in a stronger norm for an SQP method applied to the optimal control of quasilinear parabolic equations. Hehl and Neitzel \cite{HehlNeitzel2024} proved convergence in both weaker and stronger norms for time-discretized optimal control problems governed by a regularized fracture model. Gobet and Grangereau \cite{GobetGrangereau2022} considered unconstrained stochastic optimal control problems with random coefficients and cost functions. The dynamics have no diffusion term, and the drift is affine in the control. Under assumptions including joint convexity in the state and control and strong convexity in the control, they established local quadratic convergence of Newton's method in the space of uniformly bounded processes $\mathbb{H}^{\infty,\infty}(0,T;\RR^m)$.

Casas and Mateos \cite{CasasMateos2026} studied an SQP method that iterates only on the control for box-constrained optimal control problems whose first derivative has the form $J'(u(\cdot))=\kappa u(\cdot) + \Phi(u(\cdot))$, with $\kappa > 0$. They assumed second-order smoothness of the cost functional on a problem-dependent space $\LBC$, together with a compact extension of $\Phi'$ to $\LTC$, successive gains in regularity, and local Lipschitz continuity into $\LIC$. Combining these properties with a no-gap second-order sufficient optimality condition on the critical cone and strict complementarity, they proved local quadratic convergence of the control sequence in $L^q(0,T;\RR^m)$ for every $q \in [\beta, \infty]$. Their framework requires finite box constraints when $\beta > 2$. The unconstrained case corresponds to $\beta = 2$ with no finite bounds on the control, in which case the iteration reduces to Newton's method.

The convergence of SQP methods for optimal control has thus been studied extensively. Although convergence of Newton's method has also been studied, establishing local quadratic convergence for general nonlinear optimal control problems merits further study. In the Newton iteration considered here, the original nonlinear state equation is solved after each control update. In contrast, a standard SQP iteration updates the state and control simultaneously, and the updated pair generally does not satisfy the original nonlinear dynamics. Provided that the state equation is solved exactly, every state--control pair produced by Newton's method is dynamically feasible. Even if the iteration is terminated early, the computed state trajectory remains consistent with the current control; see \cref{RSQP}. This property motivates the present convergence analysis, in which the two-norm discrepancy is addressed. A further difficulty is that the second-order coercivity condition \eqref{SSC} required by Newton's method is a functional inequality in an infinite-dimensional space and is difficult to verify directly in applications.

\subsection{Main results}

First, we prove that the cost functional $J$ is twice G{\^a}teaux differentiable on $\LTC$ and derive an explicit operator representation of $J''$. We examine the different continuity properties of $J''$ in $\LIC$ and $\LTC$ and give sufficient conditions for local Lipschitz continuity in $\LTC$.
Second, we characterize second-order coercivity in $\LTC$ by the existence of a strongly regular solution to an associated matrix Riccati equation. The resulting equivalence reduces the verification of an infinite-dimensional quadratic-form condition to the solvability of a matrix differential equation.
Finally, under the second-order coercivity condition, we use linear-quadratic optimal control theory to address the two-norm discrepancy and establish local quadratic convergence of Newton's method in $\LIC$. When $b_{uu}$ and $f_{uu}$ are independent of the control, we also establish local quadratic convergence in $\LTC$. Explicit estimates for the corresponding radii of convergence are also obtained.

\subsection{Organization of the paper}

The remainder of the paper is organized as follows. Section~2 introduces the notation, assumptions, and preliminary results. Section~3 studies second-order differentiability of the cost functional, an operator representation of its second derivative, and the continuity of this derivative. Section~4 gives a Riccati characterization of second-order coercivity. Section~5 presents the Newton iteration, establishes its local quadratic convergence, and gives estimates for the radii of convergence and numerical examples.


\section{Notation and preliminaries}

\subsection{Notation}

The following notation is used throughout the paper.

$\triangleright$ {\itshape Norms and inner products.} For a normed space $X$, its norm is denoted by $\norm{\cdot}_X$. If $X$ is a Hilbert space, its inner product is denoted by $\inner{\cdot}{\cdot}_X$. In Euclidean spaces, we use $\abs{\cdot}$ and $\inner{\cdot}{\cdot}$ for the norm and inner product, respectively. For $x\in X$ and $\delta>0$, let $B_X(x,\delta)$ denote the closed ball in $X$ with center $x$ and radius $\delta$.

$\triangleright$ {\itshape Linear operators.} For normed spaces $X,Y$, let $\mathbb{L}(X,Y)$ denote the space of bounded linear operators from $X$ to $Y$, equipped with the norm
\begin{equation*}
	\norm{\varphi}_{\mathbb{L}(X,Y)}
	= \underset{x \in X \backslash \{0\}}{\sup}\frac{\norm{\varphi x}_Y}{\norm{x}_X}.
\end{equation*}
We write $\mathbb{L}(X)=\mathbb{L}(X,X)$.

$\triangleright$ {\itshape Function spaces.} For $\beta\in[1,\infty)$, let $\LB$ and $\LI$ denote the spaces of $\beta$-integrable and essentially bounded functions from $[0,T]$ to $\mathbb{R}^m$, respectively, equipped with the norms
\begin{equation*}
	\norm{\varphi}_\beta \triangleq
	\left(
		\int_{0}^{T}\abs{\varphi(t)}^\beta \dd t
	\right)^\frac{1}{\beta}, \quad
	\norm{\varphi}_\infty\triangleq 
	\underset{t \in [0,T]}{\esssup}
	\abs{\varphi(t)}.
\end{equation*}

$\triangleright$ {\itshape Derivatives.} For a function $\varphi(x_1,x_2,\cdots,x_n)$, we write $\varphi_{x_i}$ for its partial derivative with respect to $x_i$ and $\varphi_{x_ix_j}$ for the partial derivative of $\varphi_{x_i}$ with respect to $x_j$. The Hessian with respect to $(x_i,x_j)$ and its quadratic form are denoted by
\begin{equation*}
	\varphi_{(x_i,x_j)^2} \triangleq
	\left(\begin{matrix}
		\varphi_{x_ix_i} & \varphi_{x_ix_j} \\
		\varphi_{x_jx_i} & \varphi_{x_jx_j}
	\end{matrix}\right) \quad
	\text{and} \quad
	\varphi_{(x_i,x_j)^2}(y_i,y_j)^2
	= \left(\begin{matrix}
		y_i \\
		y_j \\
	\end{matrix}\right)^\top
	\varphi_{(x_i,x_j)^2}
	\left(\begin{matrix}
		y_i \\
		y_j \\
	\end{matrix}\right).
\end{equation*}
For $k\in\NN$, let $u(\cdot),\bar{u(\cdot)},u^k(\cdot)\in\LT$, and let $x(\cdot),\bar x(\cdot),x^k(\cdot)$ and $p(\cdot),\bar p(\cdot),p^k(\cdot)$ be their corresponding states and adjoints, respectively. For $t\in[0,T]$ and $i\in\{x,u,xx,xu,ux,uu,(x,u)^2\}$, we use the shorthand notation
\begin{equation}\label{simple}
	\left\{
	\begin{aligned}
		& b_i[t] = b_i(t,x(t),u(t)), \quad
		  f_i[t] = f_i(t,x(t),u(t)), \quad
		  H_i[t] = H_i(t,x(t),u(t),p(t)), \\
		& \overline{b_i}[t]
		= b_i(t,\bar{x}(t),\bar{u}(t)), \quad
		  \overline{f_i}[t]
		= f_i(t,\bar{x}(t),\bar{u}(t)), \quad
		  \overline{H_i}[t]
		= H_i(t,\bar{x}(t),\bar{u}(t),\bar{p}(t)) \\
		& b_i^k[t] = b_i(t,x^k(t),u^k(t)), \quad
		  f_i^k[t] = f_i(t,x^k(t),u^k(t)), \quad
		  H_i^k[t] = H_i(t,x^k(t),u^k(t),p^k(t)).
	\end{aligned}
	\right.
\end{equation}

\subsection{Preliminaries}

We impose the following assumptions on problem \eqref{DOCP}.
\begin{assumption}\label{A1}
	The mapping $b$ is measurable, and $b(t,\cdot,\cdot)$ is twice differentiable on $\mathbb{R}^n\times \mathbb{R}^m$ for every $t\in[0,T]$. There exists a constant $L>0$, independent of $t$, such that, for all $x,\overline{x}\in\mathbb{R}^n$ and $u,\overline{u}\in\mathbb{R}^m$,
	\begin{equation*}
		\left\{
		\begin{aligned}
			& \abs{b(t,x,u)} \le L(\abs{x}+\abs{u}+1), \quad
			\abs{b_x(t,x,u)}+\abs{b_u(t,x,u)} \le L, \\
			& \abs{b_{(x,u)^2}(t,x,u)} \le L, \quad
			\abs{
				b_{(x,u)^2}(t,x,u)-b_{(x,u)^2}(t,\bar{x},\bar{u})
			}
			\le L(\abs{x-\bar{x}}+\abs{u-\bar{u}}).
		\end{aligned}
		\right.
	\end{equation*}
\end{assumption}

\begin{assumption}\label{A2}
	The mappings $f$ and $g$ are measurable. For every $t\in[0,T]$, $f(t,\cdot,\cdot)$ is twice differentiable on $\mathbb{R}^n\times\mathbb{R}^m$, and $g$ is twice differentiable on $\mathbb{R}^n$. There exists a constant $L>0$, independent of $t$, such that, for all $x,\bar{x}\in\mathbb{R}^n$ and $u,\bar{u}\in\mathbb{R}^m$,
	\begin{equation*}
		\left\{
		\begin{aligned}
			& \abs{f(t,x,u)} \le L(\abs{x}^2+\abs{u}^2+1),~
			\abs{f_x(t,x,u)}+\abs{f_u(t,x,u)}
			\le L(\abs{x}+\abs{u}+1),	\\
			& \abs{f_{(x,u)^2}(t,x,u)} \le L,~
			\abs{
				f_{(x,u)^2}(t,x,u)-f_{(x,u)^2}(t,\bar{x},\bar{u})
			}
			\le L(\abs{x-\bar{x}}+\abs{u-\bar{u}}), \\
			& \abs{g(x)}\le L(\abs{x}^2+1),~
			\abs{g_x(x)} \le L(\abs{x}+1), \\
			& \abs{g_{xx}(x)} \le L,~
			\abs{g_{xx}(x)-g_{xx}(\bar{x})}
			\le L\abs{x-\bar{x}}.
		\end{aligned}
		\right.
	\end{equation*}
\end{assumption}
\newcommand{\assume}{\cref{A1,A2} } \newcommand{\Assume}{\cref{A1,A2} }

Under \assume, equality of mixed partial derivatives yields
\begin{equation*}
	H_{ux}(t,x,u,p)=H_{xu}(t,x,u,p)^\top, \quad \forall
	(t,x,u,p)\in [0,T]\times\mathbb{R}^n\times\mathbb{R}^m\times\mathbb{R}^n.
\end{equation*}

By standard ordinary differential equation theory, under \assume, the state equation \eqref{system} and the adjoint equation \eqref{adjoint} admit unique solutions for every control $u(\cdot) \in \LT$. 
\begin{definition}\label{D-feasible-group}
	Let $u(\cdot) \in \LT$. If $x(\cdot)$ and $p(\cdot)$ are the state and adjoint corresponding to $u(\cdot)$, respectively, then $(x(\cdot),u(\cdot),p(\cdot))$ is called an admissible triple. If $\bar{u(\cdot)}$ solves problem \eqref{DOCP}, the corresponding triple $(\bar{x}(\cdot),\bar{u(\cdot)},\bar{p}(\cdot))$ is called an optimal triple.
\end{definition}

Given $u(\cdot) \in \LT$ and its corresponding state $x(\cdot)$, define the state variation $y(\cdot) \in C([0,T];\RR^n)$ by
\begin{equation}\label{variable}
	\left\{
	\begin{aligned}
		& \dot{y}(t)
		= b_x(t,x(t),u(t)) y(t)
		+ b_u(t,x(t),u(t)) v(t), \quad t \in [0,T], \\
		& y(0) = 0.
	\end{aligned}
	\right.
\end{equation}

\begin{lemma}[{Second-order sufficient condition \cite{BonnansSilva2012}}]\label{LSSC}
	Suppose that \assume hold, and let $(\bar{x}(\cdot),\bar{u(\cdot)},\bar{p}(\cdot))$ be an admissible triple satisfying the stationarity condition \eqref{FNC}. Assume that there exists a constant $\alpha>0$ such that, for every $v(\cdot)\in \LT$,
	\begin{equation}\label{SSC}
		\int_{0}^{T}H_{(x,u)^2}(t,\bar{x}(t),\bar{u}(t),\bar{p}(t))
		(y(t),v(t))^2dt
		+ y(T)^\top g_{xx}(\bar{x}(T))y(T)
		\ge \alpha\norm{v}_2^2,
	\end{equation}
	where $y(\cdot)$ is the corresponding state variation. Then there exists $\delta>0$ such that, for all $v(\cdot)\in B_{\LI}(0,\delta)$,
	\begin{equation*}
		J(\bar{u(\cdot)}+v(\cdot)) 
		\ge J(\bar{u(\cdot)}) + \frac{\alpha}{2}\alpha\norm{v}_2^2.
	\end{equation*}
\end{lemma}



\section{Second-order analysis of the cost functional}

This section studies the second-order differentiability of the cost functional in problem \eqref{DOCP} with respect to the control and the Lipschitz continuity of its second derivative.

\subsection{First-order differentiability of the cost functional}

\begin{lemma}\label{LExpH}
	Suppose that \Assume hold. Let $(x(\cdot),u(\cdot), p(\cdot))$ and $(\bar{x}(\cdot),\bar{u(\cdot)},\bar{p}(\cdot))$ be two admissible triples. Then there exists a constant $C_1 > 0$, depending only on $T,L,x_0$, such that the following estimates hold, with \eqref{Ex}-\eqref{Ep} valid for every $t \in [0,T]$ and \eqref{EH} valid for almost every $t \in [0,T]$:
	\begin{align}
		\label{Ex}
		&	\abs{x(t)} \le C_1 (1+\norm{u}_1),
		&&	\abs{x(t)-\bar{x}(t)} \le C_1 \norm{u-\bar{u}}_1, \\
		\label{Ep}
		&	\abs{p(t)} \le C_1 (1+\norm{u}_1), \quad
		&&	\abs{p(t)-\bar{p}(t)} \le C_1 (1 + \norm{u}_1) \norm{u-\bar{u}}_1, \\
		\label{EH}
		&	\abs{H_i[t]} \le C_1 (1+\norm{u}_1), 
		&&	\abs{H_i[t]-\overline{H_{i}}[t]} \le C_1 (1 + \norm{u}_1)(\norm{u-\bar{u}}_1+\abs{u(t)-\bar{u}(t)}).
	\end{align}
	Here $i \in \{xx,xu,ux,uu\}$. More precisely, $C_1$ is given by
	\begin{equation}\label{C1}
		\left\{
		\begin{aligned}
			& C_x = \max\left\{e^{LT}(\abs{x_0}+LT),~Le^{LT}\right\}, \\
			& K_p = Le^{LT} (1+C_x) (T+1), \\
			& C_p = \max \Bigl\{ K_p,  Le^{LT} \bigl[ C_x + ( C_xT+1) (1+K_p)\bigr]\Bigr\}, \\
			& C_H = \max \{ L(C_p+C_x+C_pC_x) , L(1+C_p) \}, \\
			& C_1 = \max \{ C_x, C_p, C_H \}.
		\end{aligned}
		\right.
	\end{equation}
\end{lemma}

\begin{proof}
	Integrating the state equation \eqref{system} gives
	\begin{align*}
		\abs{x(t)}
		\le & \abs{x_0} + \int_{0}^{t} \abs{b(s,x(s),u(s))} \dd s \\
		\le & \abs{x_0}
		+ L \int_{0}^{t} (\abs{x(s)}+\abs{u(s)}+1)\dd s \\
		\le & \abs{x_0} + L (\norm{u}_1+T) + L \int_{0}^{t} \abs{x(s)}\dd s,
	\end{align*}
	Gronwall's inequality yields
	\begin{equation}\label{Ex.1}
		\abs{x(t)}
		\le e^{Lt} (\abs{x_0}+LT+L\norm{u}_1)
		\le e^{LT} (\abs{x_0}+LT) + Le^{LT}\norm{u}_1.
	\end{equation}
	Similarly, \eqref{system} and \cref{A1} give
	\begin{align*}
		\abs{x(t)-\bar{x}(t)}
		\le & \int_{0}^{t} \abs{\dot{x}(s)-\dot{\bar{x}}(s)} \dd s \\
		=	& \int_{0}^{t} \abs{b(s,x(s),u(s))-b(s,\bar{x}(s),\bar{u}(s))} \dd s \\
		\le & L \int_{0}^{t} (\abs{x(s)-\bar{x}(s)}+\abs{u(s)-\bar{u}(s)})\dd s \\
		\le & L\norm{u-\bar{u}}_1 + L \int_{0}^{t} \abs{x(s)-\bar{x}(s)}\dd s, 
	\end{align*}
	Applying Gronwall's inequality gives
	\begin{equation}\label{Ex.2}
		\abs{x(t)-\bar{x}(t)}
		\le Le^{Lt}\norm{u-\bar{u}}_1
		\le Le^{LT}\norm{u-\bar{u}}_1.
	\end{equation}
	Setting $C_x = \max\left\{e^{LT}(\abs{x_0}+LT),~Le^{LT}\right\}$ and combining \eqref{Ex.1}-\eqref{Ex.2}, we obtain
	\begin{equation}\label{Ex.3}
		\abs{x(t)} \le C_x (1+\norm{u}_1), \quad
		\abs{x(t)-\bar{x}(t)} \le C_x \norm{u-\bar{u}}_1.
	\end{equation}
	
	Under \assume, integrating \eqref{adjoint} backward and using \cref{Ex.3} gives
	\begin{align*}
		\abs{p(t)}
		\le & \abs{g_x(x(T))}
		+ \int_{t}^{T} (\abs{b_x[s]^\top p(s)} + \abs{f_x[s]})\dd s \\
		\le & L (\abs{x(T)}+1)
		+ \int_{t}^{T}
		\Bigl[
		L \abs{p(s)} + L ( \abs{x(s)} + \abs{u(s)} + 1 )
		\Bigr] \dd s \\
		\le & L ( 1 + C_x + C_x \norm{u}_1 )
		+ L \int_{0}^{T} ( 1 + C_x + C_x \norm{u}_1 + \abs{u(s)} )\dd s
		+ L \int_{t}^{T} \abs{p(s)}\dd s \\
		=	& L (1+C_x) (T+1) + L ( 1 + C_x + C_x T ) \norm{u}_1
		+ L \int_{t}^{T} \abs{p(s)} \dd s.
	\end{align*}
	Applying the backward Gronwall inequality to the preceding estimate gives
	\begin{align}
		\abs{p(t)}
		&\le L \Bigl[ (1+C_x)(T+1)+(1+C_x+C_xT)\norm{u}_1 \Bigr]
		   \left(1+L\int_t^T e^{L(s-t)}\dd s\right) \notag\\
		&= L e^{L(T-t)}\Bigl[ (1+C_x)(T+1)+(1+C_x+C_xT)\norm{u}_1 \Bigr] \notag\\
		&\le L e^{LT}\Bigl[ (1+C_x)(T+1)+(1+C_x+C_xT)\norm{u}_1 \Bigr] \notag\\
		&\le K_p(1+\norm{u}_1).\label{Ep.1}
	\end{align}
	where $K_p = Le^{LT} (1+C_x) (T+1)$. Similarly, \eqref{adjoint} and \assume give
	\begin{align*}
		\abs{p(T) - \bar{p}(T)}
		=	& \abs{g_x(x(T))-g_x(\bar{x}(T))}
		\le	L \abs{x(T)-\bar{x}(T)}
		\le L C_x \norm{u-\bar{u}}_1, \\
		\abs{\dot{p}(t)-\dot{\bar{p}}(t)}
		\le & \abs{b_x[t]^\top p(t) - \overline{b_x}[t]^\top \bar{p}(t)}
			+ \abs{f_x[t]-\overline{f_x}[t]} \\
		\overset{(a)}{\le}
			& \abs{b_x[t] - \overline{b_x}[t]} \abs{p(t)}
			+ \abs{\overline{b_x}[t]} \abs{p(t)-\bar{p}(t)}
			+ \abs{f_x[t] - \overline{f_x}[t]} \\
		\le & L\abs{p(t)-\bar{p}(t)}
			+ L ( \abs{x(t)-\bar{x}(t)} + \abs{u(t)-\bar{u}(t)} ) ( \abs{p(t)}+1 ) \\
		\le & L\abs{p(t)-\bar{p}(t)}
			+ L ( C_x\norm{u-\bar{u}}_1 + \abs{u(t)-\bar{u}(t)} ) ( \abs{p(t)}+1 ),
	\end{align*}
	The estimate in $(a)$ will be used repeatedly below. Consequently,
	\begin{align*}
		\abs{p(t)-\bar{p}(t)}
		\le & \abs{p(T)-\bar{p}(T)}
		+ \int_{t}^{T} \abs{\dot{p}(s)-\dot{\bar{p}}(s)} \dd s \\
		\le & L \int_{t}^{T} \abs{p(s)-\bar{p}(s)} \dd s
		+ L C_x\norm{u-\bar{u}}_1 \\
		&	+ L \int_{0}^{T} ( C_x\norm{u-\bar{u}}_1 + \abs{u(s)-\bar{u}(s)} ) ( \abs{p(s)}+1 ) \dd s \\
		\le & L \int_{t}^{T} \abs{p(s)-\bar{p}(s)} \dd s
		+ L \Bigl[ C_x + ( C_xT+1) (1+K_p+K_p\norm{u}_1) \Bigr] \norm{u-\bar{u}}_1,
	\end{align*}
	Applying Gronwall's inequality gives
	\begin{equation}\label{Ep.2}
		\abs{p(t)-\bar{p}(t)} \le Le^{LT} \Bigl[ C_x + ( C_xT+1) (1+K_p+K_p\norm{u}_1) \Bigr] \norm{u-\bar{u}}_1.
	\end{equation}
	Setting $C_p = \max \Bigl\{ K_p,  Le^{LT} \bigl[ C_x + ( C_xT+1) (1+K_p)\bigr]\Bigr\}$ and using \eqref{Ep.1}-\eqref{Ep.2}, we obtain
	\begin{equation}\label{Ep.3}
		\abs{p(t)} \le C_p (1+\norm{u}_1), \quad
		\abs{p(t)-\bar{p}(t)} \le C_p (1 + \norm{u}_1) \norm{u-\bar{u}}_1.
	\end{equation}
	
	Finally, we estimate the derivatives of the Hamiltonian. For any $i\in\{xx,xu,ux,uu\}$, \assume and \eqref{Ex.3}, \eqref{Ep.3} imply
	\begin{equation}\label{EH.1}
		\abs{H_i[t]} \le \abs{b_i[t]} \abs{p(t)} + \abs{f_i[t]} \le K_H(1+\norm{u}_1),
	\end{equation}
	where $K_H = L(C_p+1)$. Similarly,
	\begin{equation}\label{EH.2}
		\begin{aligned}
			& \abs{H_i[t] - \overline{H_{i}}[t]} \\
			\le & \abs{b_i[t]^\top p(t) - \overline{b_i}[t]^\top \bar{p}(t)} 
				+ \abs{f_i[t] - \overline{f_i}[t]} \\
			\le & L C_p ( 1+\norm{u}_1 ) \norm{u-\bar{u}}_1
				+ L ( C_x\norm{u-\bar{u}}_1+\abs{u(t)-\bar{u}(t)} ) 
				  ( 1+C_p+C_p\norm{u}_1 ).
		\end{aligned}
	\end{equation}
	Setting $C_H = \max \{ L(C_p+C_x+C_pC_x) , L(1+C_p) \}$ and using the preceding estimate together with \eqref{EH.1}, we obtain
	\begin{equation}\label{EH.3}
		\abs{H_i[t]} \le C_H (1+\norm{u}_1), \quad
		\abs{H_i[t]-\overline{H_{i}}[t]} \le C_H (1 + \norm{u}_1)(\norm{u-\bar{u}}_1+\abs{u(t)-\bar{u}(t)}).
	\end{equation}
	Thus, setting $C_1 = \max \{ C_x, C_p, C_H \}$, we obtain \eqref{Ex}-\eqref{C1} directly from \eqref{Ex.3}, \eqref{Ep.3}-\eqref{EH.3}.
\end{proof}

\begin{lemma}[First-order state variation {\cite[Lemma 2.2 (Chapter 3)]{YongZhou1999}}]\label{Lxy}
	Suppose that \Assume hold. Let $(x(\cdot),u(\cdot),p(\cdot))$ be an admissible triple and fix $v(\cdot) \in \LT$. Set $u^\eps = u + \eps v$, and let $x^\eps$ be the state corresponding to $u^\eps$. Then
	\begin{equation}\label{x-limit}
		\underset{\eps \to 0}{\lim} ~ \underset{t \in [0,T]}{\sup}
		\abs*{ \frac{x^\eps(t) - x(t)}{\eps} - y(t) }
		= 0,
	\end{equation}
	where $y(\cdot)$ solves the first-order variational equation \eqref{variable}.
\end{lemma}

\begin{proposition}\label{PFOD}
	Suppose that \Assume hold. Then $J$ is Fr{\'e}chet differentiable, and the Riesz representative $J'(u(\cdot)) \in \LT$ of its first derivative is given by
	\begin{equation}\label{FOD}
		J'(u(\cdot))(t) = H_u[t], \quad \forall u(\cdot) \in \LT, \text{a.e. }t\in[0,T].
	\end{equation}
\end{proposition}
\begin{proof}
	We have $H_u[\cdot] \in \LT$. For any $v(\cdot) \in \LT$, let $y(\cdot),u^\eps(\cdot),x^\eps(\cdot)$ be defined as in \cref{Lxy}. Set
	\begin{equation*}
		x^{\eps,\theta} = x + \theta (x^\eps-x), \quad u^{\eps,\theta} = u + \theta \eps v,
	\end{equation*}
	The definition of $J$ gives
	\begin{align*}
		& J(u^\eps(\cdot)) - J(u(\cdot)) \\
		=	& \int_{0}^{T}
			  \Bigl(
				f(t,x^\eps(t),u^\eps(t)) - f[t]
			  \Bigr) \dd t
			+ g(x^\eps(T)) - g(x(T)) \\
		=	& \int_{0}^{T}
			  \inner*{
				\int_{0}^{1} f_x(t,x^{\eps,\theta}(t),u^{\eps,\theta}(t)) \dd \theta
			  }{ x^\eps(t)-x(t) } \dd t
			+ \eps\int_{0}^{T}
		  		\inner*{ \int_{0}^{1} f_u(t,x^{\eps,\theta}(t),u^{\eps,\theta}(t)) \dd \theta
		  	  }{v(t)} \dd t \\
		  & + \inner*{
				\int_{0}^{1}g_x(x^{\eps,\theta}(T))\dd \theta
			  }{ (x^\eps(T)-x(T)) }.
	\end{align*}
	By \assume, \cref{Lxy}, and the Lebesgue dominated convergence theorem,
	\begin{equation}\label{PFOD.1}
		\underset{\eps \to 0}{\lim}
		\frac{ J(u^\eps(\cdot)) - J(u(\cdot)) }{\eps}
		= \int_{0}^{T}
		  \Bigl(
			\inner{f_x[t]}{y(t)}
			+ \inner{f_u[t]}{v(t)}
		  \Bigr)\dd t
		+ \inner{g_x(x(T))}{y(T)}.
	\end{equation}
	On the other hand, \eqref{variable} and \eqref{adjoint} yield
	\begin{equation}\label{PFOD.2}
		\begin{aligned}
			\int_{0}^{T} \inner{b_u[t]^\top p(t)}{v(t)} \dd t
			= & \int_{0}^{T} \inner{p(t)}{b_u[t] v(t)} \dd t \\
			= & \int_{0}^{T}\Bigl(
			\inner{p(t)}{\dot{y}(t)} - \inner{b_x[t]^\top p(t)}{y(t)}
			\Bigr) \dd t \\
			= & \int_{0}^{T}\Bigl(
			\inner{p(t)}{\dot{y}(t)} + \inner{\dot{p}(t)}{y(t)} + \inner{f_x[t]}{y(t)}
			\Bigr) \dd t \\
			= & \int_{0}^{T} \inner{f_x[t]}{y(t)} \dd t
			+ \inner{g_x(x(T))}{y(T)}.
		\end{aligned}
	\end{equation}
	Substituting \eqref{PFOD.2} into \eqref{PFOD.1} gives
	\begin{equation}\label{PFOD.3}
		\underset{\eps \to 0}{\lim}
		\frac{ J(u^\eps(\cdot)) - J(u(\cdot)) }{\eps}
		= \int_{0}^{T} \inner{H_u[t]}{v(t)} \dd t.
	\end{equation}
	Moreover, \cref{LExpH} and \assume imply
	\begin{equation}\label{EHu}
		\abs{H_u[t]}
		\le \abs{b_u[t]} \abs{p(t)} + \abs{f_u[t]}
		\le L [ (C_x+C_p)( 1+\norm{u}_1 ) + \abs{u(t)} + 1 ],
	\end{equation}
	This shows that $H_u[\cdot] \in \LT$. Hence, by \eqref{PFOD.3}, $J$ is G{\^a}teaux differentiable, and its G{\^a}teaux derivative has Riesz representative $J'(u(\cdot)) \in \LT$.
	
	We next prove that $J'$ is continuous. Let $(\bar{x}(\cdot),\bar{u(\cdot)},\bar{p}(\cdot))$ be another admissible triple. By \eqref{PFOD.3} and \assume,
	\begin{align*}
		\underset{\norm{u-\bar{u}}_2 \to 0}{\lim}
		\norm{ J'(u(\cdot)) - J'(\bar{u(\cdot)}) }_2^2
		=	& \underset{\norm{u-\bar{u}}_2 \to 0}{\lim}
			  \int_{0}^{T}\abs{ H_u[t] - \overline{H_u}[t] }^2 \dd t \\
		\overset{(a)}{\le}
			& \underset{\norm{u-\bar{u}}_2 \to 0}{\lim}
			  \int_{0}^{T} C_1^2 ( 1+\norm{u}_1 )^2
			  ( \norm{u-\bar{u}}_1 + \abs{u(t) - \bar{u}(t)} )^2 \dd t \\
		\overset{(b)}{\le}
			& \underset{\norm{u-\bar{u}}_2 \to 0}{\lim}
			  C_1^2 ( 1+\norm{u}_1 )^2 ( T+1 )^2 \norm{u-\bar{u}}_2^2 \\
		=	& 0.
	\end{align*}
	Here $(a)$ is obtained as in the derivation of \eqref{EH.2}, and $(b)$ follows from H{\"o}lder's inequality. The continuity of the G{\^a}teaux derivative implies that $J$ is Fr{\'e}chet differentiable.
\end{proof}

\subsection{Second-order differentiability of the cost functional}

For given $u(\cdot),v(\cdot) \in \LT$, let $y(\cdot)$ solve \eqref{variable}, and define the adjoint variation $\psi(\cdot) \in C([0,T];\RR^n)$ by
\begin{equation}\label{psi}
	\left\{
	\begin{aligned}
		& \dot{\psi}(t)
		= - b_x[t]^\top\psi(t) - H_{xx}[t]y(t) - H_{xu}[t]v(t), \quad \text{a.e. }t\in[0,T], \\
		& \psi(T) = g_{xx}(x(T))y(T).
	\end{aligned}
	\right.
\end{equation}
Equations \eqref{variable} and \eqref{psi} show that $y(\cdot)$ and $\psi(\cdot)$ depend linearly on the direction $v(\cdot)$.

\begin{lemma}\label{LEypsi}
	Suppose that \Assume hold. Let $(x(\cdot),u(\cdot),p(\cdot))$ and $(\bar{x}(\cdot),\bar{u(\cdot)},\bar{p}(\cdot))$ be admissible triples, and fix $v(\cdot) \in \LT$. Denote the corresponding state variations by $y(\cdot),\bar{y}(\cdot)$ and the adjoint variations by $\psi(\cdot),\bar{\psi}(\cdot)$. Then there exists a constant $C_2>0$, depending only on $T,L,x_0$, such that, for every $t \in [0,T]$,
	\begin{align}
		\label{Ey}
		&	\abs{y(t)} \le C_2\norm{v}_1, \quad
		&&	\abs{y(t)-\bar{y}(t)} \le C_2 \norm{u-\bar{u}}_2 \norm{v}_2, \\
		\label{Epsi}
		&	\abs{\psi(t)} \le C_2(1+\norm{u}_1) \norm{v}_1, \quad
		&&	\abs{\psi(t)-\bar{\psi}(t)} \le C_2 (1+\norm{u}_1) \norm{u-\bar{u}}_2 \norm{v}_2.
	\end{align}
	More precisely, $C_2$ is given by
	\begin{equation}\label{C2}
		\left\{
		\begin{aligned}
			& C_{y}
			= Le^{LT} \left[ ( C_1 + Le^{LT} + C_1LT e^{LT} )T + 1 \right], \\
			& K_\psi
			= e^{LT} [ LC_y + C_1 ( C_yT + 1 ) ], \\
			& C_\psi
			= \max\Bigl\{K_\psi,e^{LT}\bigl[LC_y(C_1T+1)+LK_\psi T(C_1T+1) \\
			&\hspace{9em}+C_1(C_yT^2+2C_yT+T+1)\bigr]\Bigr\}, \\
			& C_2 = \max \{ C_y, C_\psi \}.
		\end{aligned}
		\right.
	\end{equation}
\end{lemma}
\begin{proof}
	Integrating \eqref{variable} and using \cref{A1}, we obtain
	\begin{equation*}
		\abs{y(t)}
		\le \int_{0}^{t}\abs{\dot{y}(s)}\dd s
		\le \int_{0}^{t}(
		\abs{b_x[s]} \abs{y(s)}
		+ \abs{b_u[s]}\abs{v(s)}
		) \dd s \\
		\le L \int_{0}^{t} \abs{y(s)} \dd s + L\norm{v}_1,
	\end{equation*}
	Gronwall's inequality therefore gives
	\begin{equation}\label{Ey.1}
		\abs{y(t)} \le Le^{Lt}\norm{v}_1 \le Le^{LT}\norm{v}_1.
	\end{equation}
	Similarly,
	\begin{align*}
		& \abs{\dot{y}(t)-\dot{\bar{y}}(t)} \\
		\le & \abs{b_x[t] y(t) - \overline{b_x}[t] \bar{y}(t)}
			+ \abs{b_u[t]- \overline{b_u}[t]} \abs{v(t)} \\
		\le & \abs{b_x[t]} \abs{y(t)-\bar{y}(t)}
			+ \abs{b_x[t] - \overline{b_x}[t]} \abs{\bar{y}(t)}
			+ L (
				\abs{x(t)-\bar{x}(t)}
				+ \abs{u(t)-\bar{u}(t)}
			  ) \abs{v(t)} \\
		\le & L \abs{y(t)-\bar{y}(t)}
			+ L ( \abs{x(t)-\bar{x}(t)} + \abs{u(t)-\bar{u}(t)} )
			  ( \abs{\bar{y}(t)} + \abs{v(t)} ) \\
		\le & L \abs{y(t)-\bar{y}(t)}
			+ L ( C_1\norm{u-\bar{u}}_1 + \abs{u(t)-\bar{u}(t)} )
			  ( Le^{LT}\norm{v}_1 + \abs{v(t)} ),
	\end{align*}
	H{\"o}lder's inequality then gives
	\begin{align*}
		& \abs{y(t)-\bar{y}(t)} \\
		\le & L \int_{0}^{T}
			  ( C_1\norm{u-\bar{u}}_1+\abs{u(s)-\bar{u}(s)} ) 
			  ( Le^{LT}\norm{v}_1+\abs{v(s)} ) \dd s
			+ L \int_{0}^{t} \abs{y(s)-\bar{y}(s)} \dd s \\
		\le & L
			  \left[
				( C_1 + Le^{LT} + C_1LT e^{LT} ) \norm{u-\bar{u}}_1 \norm{v}_1
				+ \norm{u-\bar{u}}_2 \norm{v}_2
			  \right]
			+ L \int_{0}^{t} \abs{y(s)-\bar{y}(s)} \dd s,
	\end{align*}
	Applying Gronwall's inequality gives
	\begin{equation}\label{Ey.3}
		\abs{y(t)} \le C_y\norm{v}_1, \quad
		\abs{y(t)-\bar{y}(t)} \le C_y \norm{u-\bar{u}}_2 \norm{v}_2,
	\end{equation}
	where $C_{y} = Le^{LT} \left[ ( C_1 + Le^{LT} + C_1LT e^{LT} )T + 1 \right]$.
	
	By \assume and \eqref{Ey.3},
	\begin{align*}
		\abs{\psi(t)}
		\le & \abs{g_{xx}(x(T))} \abs{y(T)}
			+ \int_{t}^{T}
			  \Bigl(
				\abs{b_x[s]}\abs{\psi(s)}
				+ \abs{H_{xx}[s]}\abs{y(s)}
				+ \abs{H_{xu}[s]}\abs{v(s)}
			  \Bigr) \dd s \\
		\le & LC_{y}\norm{v}_1
			+ L \int_{t}^{T} \abs{\psi(s)} \dd s
			+ C_1 \int_{0}^{T}
			  ( 1 + \norm{u}_1) ( C_{y}\norm{v}_1+\abs{v(s)} ) \dd s \\
		\le & [ C_1 ( 1 + \norm{u}_1 ) ( C_yT + 1 ) + LC_y ] \norm{v}_1
			+ L \int_{t}^{T} \abs{\psi(s)} \dd s,
	\end{align*}
	Applying Gronwall's inequality gives
	\begin{equation}\label{Epsi.1}
		\abs{\psi(t)}
		\le K_\psi ( 1 + \norm{u}_1 )\norm{v}_1,
	\end{equation}
	where $K_\psi = e^{LT} \left[ C_1(C_{y}T+1) + LC_{y} \right]$.
	Similarly,
	\begin{align*}
		\abs{\psi(T)-\bar{\psi}(T)} 
		\le & \abs{g_{xx}(x(T))-g_{xx}(\bar{x}(T))} \abs{y(T)}
			+ \abs{g_{xx}(\bar{x}(T))} \abs{y(T) - \bar{y}(T)} \\
		\le & LC_1C_y \norm{u-\bar{u}}_1 \norm{v}_1
			+ LC_y\norm{u-\bar{u}}_2 \norm{v}_2 \\
		\le & LC_y ( C_1T + 1 ) \norm{u-\bar{u}}_2 \norm{v}_2, \\
			  \abs{ \dot{\psi}(t)- \dot{\bar{\psi}}(t) }
		\le & \abs{ b_x[t]^\top\psi(t) - \overline{b_x}[t]^\top \bar{\psi}(t) }
			+ \abs{ H_{xx}[t]y(t) - \overline{H_{xx}}[t] \bar{y}(t) }
			+ \abs{ H_{xu}[t] - \overline{H_{xu}}[t] } \abs{v(t)} \\
		\le & \abs{ b_x[t] - \overline{b_x}[t] } \abs{\psi(t)}
			+ \abs{\overline{b_x}[t]} \abs{\psi(t)-\bar{\psi}(t)}
			+ \abs{ H_{xx}[t]} \abs{y(t)-\bar{y}(t)} \\
		  & + \abs{ H_{xx}[t] - \overline{H_{xx}}[t]} \abs{\bar{y}(t)}
			+ \abs{ H_{xu}[t] - \overline{H_{xu}}[t] }\abs{v(t)} \\
		\le & \Bigl\{
				C_1C_{y}\norm{u-\bar{u}}_2\norm{v}_2
				+ C_1\bigl( \norm{u-\bar{u}}_1 + \abs{u(t)-\bar{u}(t)} \bigr)
				  \bigl( C_y\norm{v}_1+\abs{v(t)} \bigr) \\
		  &		+ LK_\psi
				  \bigl( C_1\norm{u-\bar{u}}_1 + \abs{u(t)-\bar{u}(t)} \bigr)
				  \norm{v}_1
			  \Bigr\} (1+\norm{u}_1)
			+ L\abs{\psi(t)-\bar{\psi}(t)},
	\end{align*}
	H{\"o}lder's inequality gives
	\begin{align*}
		& \abs{\psi(t)-\bar{\psi}(t)} \\
		\le & LC_1C_y \norm{u-\bar{u}}_1\norm{v}_1
		+ LC_y\norm{u-\bar{u}}_2\norm{v}_2
		+ L \int_{t}^{T} \abs{\psi(s)-\bar{\psi}(s)} \dd s \\
		& + \int_{0}^{T}\Bigl\{
		C_1C_{y}\norm{u-\bar{u}}_2\norm{v}_2
		+ C_1 \bigl( \norm{u-\bar{u}}_1 + \abs{u(s)-\bar{u}(s)} \bigr)
		\bigl( C_y\norm{v}_1+\abs{v(s)} \bigr) \\
		&	\qquad
		+ LK_\psi \bigl( C_1\norm{u-\bar{u}}_1 + \abs{u(s)-\bar{u}(s)} \bigr)
		\norm{v}_1
		\Bigr\} (1+\norm{u}_1) \dd s \\
		\le & \Bigl\{ (1+\norm{u}_1) \Bigl[ C_1(C_yT+C_y+1) + LK_\psi(C_1T+1) \Bigr] + LC_1C_y \Bigr\}
		\norm{u-\bar{u}}_1\norm{v}_1 \\
		&	+ \bigl[C_1(C_yT+1)(1+\norm{u}_1)+LC_y\bigr] \norm{u-\bar{u}}_2\norm{v}_2
		+ L \int_{t}^{T} \abs{\psi(s)-\bar{\psi}(s)} \dd s \\
		\le & (1+\norm{u}_1)\Bigl\{ T\Bigl[ C_1(C_yT+C_y+1) + LK_\psi(C_1T+1) \Bigr] \\
		& \qquad + LC_y(C_1T+1)+C_1(C_yT+1) \Bigr\}
		\norm{u-\bar{u}}_2\norm{v}_2 \\
		& +  L \int_{t}^{T} \abs{\psi(s)-\bar{\psi}(s)} \dd s,
	\end{align*}
	Applying Gronwall's inequality gives
	\begin{equation}\label{Epsi.2}
		\begin{aligned}
		\abs{\psi(t)-\bar{\psi}(t)}
		\le {}& e^{LT}
			\Bigl\{
				(1+\norm{u}_1)
				\Bigl[ LC_y(C_1T+1)+LK_\psi T(C_1T+1) \\
		& \qquad + C_1(C_yT^2+2C_yT+T+1)\Bigr]
			\Bigr\}
		\norm{u-\bar{u}}_2\norm{v}_2.
		\end{aligned}
	\end{equation}
	Setting $C_\psi = \max\{K_\psi,e^{LT}[LC_y(C_1T+1)+LK_\psi T(C_1T+1)+C_1(C_yT^2+2C_yT+T+1)]\}$ and combining \eqref{Epsi.1}-\eqref{Epsi.2}, we obtain
	\begin{equation}\label{Epsi.3}
		\abs{\psi(t)} \le C_\psi(1+\norm{u}_1) \norm{v}_1, \quad
		\abs{\psi(t)-\bar{\psi}(t)} \le C_\psi (1+\norm{u}_1) \norm{u-\bar{u}}_2 \norm{v}_2
	\end{equation}
	Set
	$C_2=\max \{ C_{y}, C_{\psi} \}$. Combining \eqref{Ey.3} and \eqref{Epsi.3}, and noting that $t\in[0,T]$ is arbitrary, yields \eqref{Ey} and \eqref{Epsi}.
\end{proof}

\begin{lemma}\label{Lppsi}
	Suppose that \Assume hold. Let $(x(\cdot),u(\cdot),p(\cdot))$ be an admissible triple and fix $v(\cdot) \in \LT$. Set $u^\eps = u + \eps v$, let $x^\eps(\cdot),p^\eps(\cdot)$ be the state and adjoint corresponding to $u^\eps(\cdot)$, and let $y(\cdot), \psi(\cdot)$ be defined by \eqref{variable} and \eqref{psi}, respectively. Then
	\begin{equation}\label{ppsi}
		\underset{\eps \to 0}{\lim} \underset{t \in [0,T]}{\sup}
		\abs*{ \frac{ p^\eps(t)-p(t) }{\eps} - \psi(t) }
		= 0.
	\end{equation}
\end{lemma}

\begin{proof}
	For $\eps \ne 0$, set
	\begin{equation}\label{simple.1}
		X^\eps = \frac{x^\eps - x}{\eps}, \quad
		P^\eps = \frac{p^\eps - p}{\eps}, \quad
		x^{\eps,\theta} = x + \theta (x^\eps-x), \quad
		u^{\eps,\theta} = u + \theta \eps v, \quad
		p^{\eps,\theta} = p + \theta (p^\eps - p),
	\end{equation}
	For $i \in \{ u, xx,xu,ux,uu \}$, write
	\begin{equation}\label{simple.2}
		H_{i}^{\eps,\theta}[t]
		= H_{i}( t, x^{\eps,\theta}(t), u^{\eps,\theta}(t), p^{\eps,\theta}(t) ).
	\end{equation}
	Also write $b_j^{\eps,\theta}[t]=b_j(t,x^{\eps,\theta}(t),u^{\eps,\theta}(t))$, $j\in\{x,u\}$. Equation \eqref{adjoint} gives
	\begin{equation}\label{ppsi.1}
		P^\eps(T)
		= \frac{
			g_x(x^\varepsilon(T))-g_x(x(T))
		  }{\eps}
		= \int_{0}^{1}g_{xx}( x^{\eps,\theta}(T) )
		  X^\varepsilon(T) \dd \theta
	\end{equation}
	and
	\begin{equation}\label{ppsi.2}
		\begin{aligned}
			\dot{P^\eps}(t)
			= & - \frac{
					H_x ( t,x^\eps(t),u^\eps(t),p^\eps(t) )
					- H_x[t]
				  }{\eps} \\
			= & - \int_{0}^{1}
				  \Bigl[
					H_{xx}(t,x^{\eps,\theta}(t),u^{\eps,\theta}(t),p^{\eps,\theta}(t))
					X^\eps(t)
					+ H_{xu}(t,x^{\eps,\theta}(t),u^{\eps,\theta}(t),p^{\eps,\theta}(t))
					v (t) \\
			  &	\quad
			  		+ b_{x}(t,x^{\eps,\theta}(t),u^{\eps,\theta}(t))^\top P^\eps(t)
				  \Bigr] \dd \theta.
		\end{aligned}
	\end{equation}
	By \cref{LExpH}, \cref{LEypsi}, \eqref{ppsi.1}, and \eqref{ppsi.2},
	\begin{equation}\label{ppsi.3}
		\begin{aligned}
			& \abs{ \psi(T) - P^\eps(T) } \\
			=	& \abs*{
					g_{xx}(x(T))y(T)
					- \int_{0}^{1}g_{xx}( x^{\eps,\theta}(T) )
					  X^\varepsilon(T) \dd \theta
				  } \\
			\le & \int_{0}^{1} \Bigl[
					\abs{
						g_{xx}(x(T)) - g_{xx}( x^{\eps,\theta}(T) )
					} \abs{ y(T) }
					+ \abs{ g_{xx}( x^{\eps,\theta}(T) ) }
					  \abs{ y(T) - X^\eps(T) }
				  \Bigr] \dd \theta \\
			\le & L \abs{ x(T) - x^\eps(T) } \abs{ y(T) }
				+ L \abs{ y(T) - X^\eps(T) } \\
			\le & \abs{\eps} L C_1C_2 \norm{v}_1^2
				+ L \abs{ y(T) - X^\eps(T) }
		\end{aligned}
	\end{equation}
	and
	\begin{align}
			& \abs{ \dot{\psi}(t) - \dot{P^\eps}(t)} \notag \\*
			\le & \int_{0}^{1}
				  \Big[
					\abs{
						H_{xx}^{\eps,\theta}[t] X^\eps(t) - H_{xx}[t] y(t)
					}
					+ \abs{
						H_{xu}^{\eps,\theta}[t] - H_{xu}[t] 
					} \abs{v(t)}
				  \Big] \dd \theta \notag \\
			  & + \int_0^1\abs{
			  		b_x(t,x^{\eps,\theta}(t),u^{\eps,\theta}(t))^\top P^\eps(t)
			  		- b_x[t]^\top\psi(t)
			  	  }\dd\theta \notag \\
			\le & \int_{0}^{1}\Bigl[
					\abs{ H_{xx}^{\eps,\theta}[t] - H_{xx}[t] }
					\abs{ X^\eps(t) }
					+ \abs{ H_{xx}[t] } \abs{ X^\eps(t) - y(t) }
					+ \abs{
						H_{xu}^{\eps,\theta}[t] - H_{xu}[t] 
					  } \abs{v(t)} \notag \\
				&	+ \abs{
			  			b_x(t,x^{\eps,\theta}(t),u^{\eps,\theta}(t))^\top P^\eps(t)
			  			- b_x[t]^\top\psi(t)
			  	  	  }
			  	  \Bigr] \dd \theta \notag \\
			\le & \int_{0}^{1}\Bigl[
					L\Big(
					  (
						\abs{ x^{\eps,\theta}(t) - x(t) }
						+ \abs{ u^{\eps,\theta}(t) - u(t) }
					  ) ( \abs{p(t)} + 1 )
					  + \abs{ p^{\eps,\theta}(t) - p(t) }
					\Big) ( \abs{X^\eps(t)} + \abs{v(t)} ) \notag \\
				&	+ C_1( 1 + \norm{u}_1 ) \abs{ X^\eps(t) - y(t) }
					+ L (
						\abs{ x^{\eps,\theta}(t) - x(t) }
						+ \abs{ u^{\eps,\theta}(t) - u(t) }
				  	  ) \abs{P^\eps(t)}
				  \Bigr] \dd \theta \notag \\
			  & + L \abs{ \psi(t) - P^\eps(t) } \notag \\
			\le & \abs{\eps} L
				  \Big[
					( C_1\norm{v}_1 + \abs{v(t)} ) (1 + C_1 + C_1\norm{u}_1)
					+ C_1( 1 + \norm{u}_1 ) \norm{v}_1
				  \Big] ( C_1\norm{v}_1 + \abs{v(t)} ) \notag \\
			  & + \abs{\eps} LC_1 ( 1 + \norm{u}_1 )
				  ( C_1\norm{v}_1 + \abs{v(t)} ) \norm{v}_1 \notag \\
			  & + C_1( 1 + \norm{u}_1 ) \abs{ X^\eps(t) - y(t) }
				+ L \abs{ \psi(t) - P^\eps(t) }.\label{ppsi.4}
	\end{align}
	Combining \cref{Lxy}, \eqref{ppsi.3}, and \eqref{ppsi.4} with Gronwall's inequality gives
	\begin{equation*}
		\underset{\eps \to 0}{\lim} \underset{t \in [0,T]}{\sup}
		\abs{ \psi(t) - P^\eps(t) }
		= 0,
	\end{equation*}
	which proves \eqref{ppsi}.
\end{proof}

\begin{theorem}[Second-order differentiability]\label{TSOD}
	Suppose that \Assume hold. Let $(x(\cdot),u(\cdot),p(\cdot))$ be an admissible triple, let $v(\cdot) \in \LT$, and let $y(\cdot),\psi(\cdot)$ denote the corresponding state and adjoint variations, respectively. Then the following assertions hold.
	\begin{enumerate}
		\item $J$ is twice G{\^a}teaux differentiable, and its second derivative $J''(u(\cdot)) \in \mathbb{L}(\LT)$ is given by
		\begin{equation}\label{SOD}
			(J''(u(\cdot))v(\cdot))(t)
			= H_{ux}[t]y(t)
			+ H_{uu}[t]v(t)
			+ b_u[t]^\top \psi(t), \quad
			\text{a.e. }t\in [0,T].
		\end{equation}
		
		\item For any $w(\cdot) \in \LT$, let $y_w(\cdot)$ be the state variation corresponding to the direction $w(\cdot)$. The second derivative $J''(u(\cdot))$ of $J$ satisfies
		\begin{equation}\label{SOD-represation}
			\inner{J''(u(\cdot))v}{w}_{\LT}
			= \int_{0}^{T}
			  \left(\begin{matrix}
				y_w(t) \\ w(t)
			  \end{matrix}\right)^\top
			  H_{(x,u)^2}[t]
			  \left(\begin{matrix}
				y(t) \\ v(t)
			  \end{matrix}\right) \dd t
			+ y_w(T)^\top g_{xx}(x(T)) y(T).
		\end{equation}
		
		\item If $u(\cdot) \in \LB (\beta \in [2,\infty])$, then $J''(u(\cdot)) \in \mathbb{L}(\LB)$.
	\end{enumerate}
\end{theorem}

\begin{proof}
	(i) By \cref{PFOD}, $J$ is Fr{\'e}chet differentiable, with derivative given by \eqref{FOD}. Let $u^\eps,x^\eps,p^\eps$ be as in \cref{Lppsi}, and retain the notation in \eqref{simple.1} and \eqref{simple.2}.
	First, as in the derivation of \eqref{ppsi.2}, we obtain
	\begin{equation}\label{SOD.1}
		\frac{
			H_u( \cdot, x^\eps(\cdot), u^\eps(\cdot), p^\eps(\cdot) )
			- H_u[\cdot]
		}{\eps} \\
		= \int_{0}^{1}
		  \Bigl[
			b_{u}^{\eps,\theta}[\cdot]^\top P^{\eps}(\cdot)
			+ H_{ux}^{\eps,\theta}[\cdot] X^\eps(\cdot)
			+ H_{uu}^{\eps,\theta}[\cdot] v(\cdot)
		  \Bigr] \dd \theta.
	\end{equation}
	Next, as in the derivation of \eqref{ppsi.4}, we obtain
	\begin{align*}
		& \norm*{
			\int_{0}^{1} b_{u}^{\eps,\theta}[\cdot]^\top P^{\eps}(\cdot) \dd \theta
			- b_u[\cdot]^\top \psi(\cdot)
		  }_2 \\
		\le	& \norm*{
				  \int_{0}^{1}
				  \abs{
					b_{u}^{\eps,\theta}[t]^\top P^{\eps}(t)
					- b_u[t]^\top \psi(t)
				  } \dd \theta
			  }_2 \\
		\le & L\norm*{
				  \abs{\eps} C_1 ( 1 + \norm{u}_1 )
				  ( C_1\norm{v}_1 + \abs{v(\cdot)} ) \norm{v}_1
				  + \abs{ \psi(\cdot) - P^\eps(\cdot) } 
			  }_2 \\
		\overset{(a)}{\le}
			& L\left[
				\abs{\eps} C_1 ( 1 + \norm{u}_1 )
				( C_1\sqrt{T}\norm{v}_1 + \norm{v}_2 ) \norm{v}_1
				+ \norm{ \psi - P^\eps }_2
			  \right],
	\end{align*}
	Here $(a)$ follows from Jensen's inequality and the triangle inequality. Hence, \cref{Lppsi} gives
	\begin{equation}\label{SOD.2}
		\underset{\eps \to 0}{\lim}
		\norm*{
			\int_{0}^{1} b_{u}^{\eps,\theta}[\cdot]^\top P^{\eps}(\cdot) \dd \theta
			- b_u[\cdot]^\top \psi(\cdot)
		}_2 = 0.
	\end{equation}
	Similarly,
	\begin{align*}
		& \norm*{
			\int_{0}^{1} H_{ux}^{\eps,\theta}[\cdot] X^\eps(\cdot) \dd \theta
			- H_{ux}[\cdot] y(\cdot)
		}_2 \\
		\le & \norm*{
				\int_{0}^{1}
				\abs{ H_{ux}^{\eps,\theta}[\cdot] X^\eps(\cdot)
					  - H_{ux}[\cdot] y(\cdot) }
				\dd \theta
			  }_2 \\
		\le & \norm*{
				\abs{\eps} LC_1
				\Big[
					( C_1\norm{v}_1 + \abs{v(\cdot)} ) (1 + C_1 + C_1\norm{u}_1)
					+ C_1( 1 + \norm{u}_1 ) \norm{v}_1
				\Big] \norm{v}_1
			  }_2 \\
		  & + C_1( 1 + \norm{u}_1 ) \norm{ X^\eps - y }_2 \\
		\le & \abs{\eps} LC_1
			  \Big[
			  	( C_1\sqrt T\norm{v}_1 + \norm{v}_2 ) (1 + C_1 + C_1\norm{u}_1)
			  	+ C_1\sqrt T( 1 + \norm{u}_1 ) \norm{v}_1
			  \Big] \norm{v}_1 \\
		  & + C_1( 1 + \norm{u}_1 ) \norm{ X^\eps - y }_2,
	\end{align*}
	By Lemma \ref{Lxy},
	\begin{equation}\label{SOD.3}
		\underset{\eps \to 0}{\lim}
		\norm*{
			\int_{0}^{1} H_{ux}^{\eps,\theta}[\cdot] X^\eps(\cdot) \dd \theta
			- H_{ux}[\cdot] y(\cdot)
		}_2 = 0.
	\end{equation}
	For the last term in the integral on the right-hand side of \eqref{SOD.1}, H{\"o}lder's inequality gives
	\begin{equation}\label{SOD.4}
		\norm*{
			\int_{0}^{1} H_{uu}^{\eps,\theta}[\cdot] v(\cdot) \dd \theta
			- H_{uu}[\cdot] v(\cdot)
		}_2 \\
		\le \norm*{
				\left(
				  \int_{0}^{1}
				  \abs{ H_{uu}^{\eps,\theta}[\cdot] - H_{uu}[\cdot] }
				  \dd \theta
				\right)^2 \abs{ v(\cdot) }^2 
			}_1^\frac{1}{2}.
	\end{equation}
	By \cref{LExpH} and the boundedness of $p^{\eps,\theta}$, for $0<\abs{\eps}\le1$ there exists a constant $K > 0$, independent of $\eps,\theta,t$, such that
	\begin{equation}\label{SOD.5}
		\begin{aligned}
			& \left(
				\int_{0}^{1}
				\abs{ H_{uu}^{\eps,\theta}[t] - H_{uu}[t] }
				\dd \theta
			  \right)^2 \abs{ v(t) }^2 \\
			\le & \left[
					\int_{0}^{1}
					\abs{ H_{uu}^{\eps,\theta}[t] }
					+ \abs{H_{uu}[t]}
					\dd \theta
				  \right]^2 \abs{ v(t) }^2 \\
			\le & \left[
					\int_{0}^{1}
					L(\abs{p^{\eps,\theta}(t)} + 1)
					\dd \theta
					+ C_1 (1 + \norm{u}_1)
				  \right]^2 \abs{v(t)}^2 \\
			\le & K (1 + \norm{u}_1)^2 \abs{v(t)}^2, \quad \text{a.e. }t\in [0,T]
		\end{aligned}
	\end{equation}
	and
	\begin{equation}\label{SOD.6}
	\begin{aligned}
		& \underset{\eps \to 0}{\lim}
		  \left(
			\int_{0}^{1}
		  	\abs{ H_{uu}^{\eps,\theta}[t] - H_{uu}[t] }
		  	\dd \theta
		  \right)^2 \abs{ v(t) }^2 \\
		\le & \underset{\eps \to 0}{\lim}
			  \int_{0}^{1}
			  \abs{ H_{uu}^{\eps,\theta}[t] - H_{uu}[t] }^2
			  \dd \theta \abs{ v(t) }^2 \\
		\le & \underset{\eps \to 0}{\lim}
			  \eps^2 L^2
			  \Big[
			  	( C_1\norm{v}_1 + \abs{v(t)} ) (1 + C_1 + C_1\norm{u}_1)
			  	+ C_1( 1 + \norm{u}_1 ) \norm{v}_1
			  \Big]^2 \abs{v(t)}^2 \\
		=	& 0.
	\end{aligned}
	\end{equation}
	The Lebesgue dominated convergence theorem and \eqref{SOD.4}-\eqref{SOD.6} therefore give
	\begin{equation}\label{SOD.7}
		\underset{\eps \to 0}{\lim}
		\norm*{
			\int_{0}^{1} H_{uu}^{\eps,\theta}[\cdot] v(\cdot) \dd \theta
			- H_{uu}[\cdot] v(\cdot)
		}_2 = 0.
	\end{equation}
	Combining \eqref{SOD.1}-\eqref{SOD.3} and \eqref{SOD.7}, we obtain
	\begin{equation}\label{SOD.8}
		\underset{\eps \to 0}{\lim}
		\norm*{
			\frac{
				H_u( \cdot,x^\eps(\cdot),u^\eps(\cdot),p^\eps(\cdot) )
				- H_u[\cdot]
			}{\eps}
			- ( b_u[\cdot]^\top \psi(\cdot)
				+ H_{ux}[\cdot] y(\cdot)
				+ H_{uu}[\cdot] v(\cdot) )
		}_2
		= 0,
	\end{equation}
	Since $y(\cdot)$ and $\psi(\cdot)$ depend linearly on $v(\cdot)$, the limit above defines a linear operator $D^2J(u(\cdot)): \LT \to \LT$ for each control $u(\cdot)$ by
	\begin{equation}\label{SOD.9}
		\underset{\eps\to0}{\lim}
		\frac{
			H_u(
			\cdot,x^\eps(\cdot),u^\eps(\cdot),p^\eps(\cdot)
			)
			- H_u[\cdot]
		}{\eps}
		= D^2J(u(\cdot))v(\cdot).
	\end{equation}
	Combining \eqref{SOD.8}-\eqref{SOD.9} and using \cref{A1} and \cref{LExpH,LEypsi}, we obtain
	\begin{equation}\label{SOD.10}
		\begin{aligned}
			& \norm{ D^2J( u(\cdot) ) v(\cdot)}_2 \\
			\le & \norm{ b_u[\cdot]^\top \psi(\cdot)}_2
				+ \norm{
					H_{ux}[\cdot]y(\cdot)
				  }_2
				+ \norm{
					H_{uu}[\cdot]v(\cdot)
				  }_2 \\
			\le & LC_2\sqrt{T} ( 1 + \norm{u}_1 ) \norm{v}_1
				+ C_1C_2\sqrt{T} ( 1 + \norm{u}_1 ) \norm{v}_1
				+ C_1 ( 1 + \norm{u}_1 ) \norm{v}_2 \\
			\le	& ( 1 + \norm{u}_1 )\left[
					TC_2 ( C_1 + L ) + C_1
				  \right]\norm{v}_2,
		\end{aligned}
	\end{equation}
	Thus, $D^2J(u(\cdot))$ is bounded. It follows from \eqref{SOD.8}-\eqref{SOD.10} that $J'$ is G{\^a}teaux differentiable, that is, $J$ is twice G{\^a}teaux differentiable, and $J''(u(\cdot)) \in \mathbb{L}( \LT)$ is given by \eqref{SOD}.
	
	(ii) Let $w(\cdot) \in \LT$, and let $y_w(\cdot)$ be the corresponding state variation. Equations \eqref{variable} and \eqref{psi} give
	\begin{equation}\label{SOD.11}
		\begin{aligned}
			& \int_{0}^{T}\inner{b_u[t]^\top \psi(t)}{w(t)} \dd t \\
			=	& \int_{0}^{T} \inner{\psi(t)}{b_u[t] w(t)} \dd t \\
			=	& \int_{0}^{T} \inner{\psi(t)}{\dot{y_w}(t)} \dd t
				- \int_{0}^{T} \inner{\psi(t)}{b_x[t] y_w(t)} \dd t \\
			=	& \int_{0}^{T} \inner{\psi(t)}{\dot{y_w}(t)} \dd t
				- \int_{0}^{T} \inner{b_x[t]^\top \psi(t)}{y_w(t)} \dd t \\
			=	& \int_{0}^{T} \inner{\psi(t)}{\dot{y_w}(t)} \dd t
				+ \int_{0}^{T} \inner{\dot{\psi}(t)}{y_w(t)} \dd t
				+ \int_{0}^{T} \inner{H_{xx}[t] y(t) + H_{xu}[t] v(t)}{y_w(t)} \dd t \\
			=	& \int_{0}^{T} \inner{H_{xx}[t] y(t) + H_{xu}[t] v(t)}{y_w(t)} \dd t
				+ \inner{g_{xx}(x(T)) y(T)}{y_w(T)}.
		\end{aligned}
	\end{equation}
	Combining \eqref{SOD.11} and \eqref{SOD}, we obtain
	\begin{equation*}
		\inner{J''(u(\cdot)) v(\cdot)}{w(\cdot)}_{\LT} \\
		= \int_{0}^{T}
		  \left(\begin{matrix}
			y_w(t) \\ w(t)
		  \end{matrix}\right)^\top
		  H_{(x,u)^2}[t]
		  \left(\begin{matrix}
			y(t) \\ v(t)
		  \end{matrix}\right) \dd t
		+ y_w(T)^\top g_{xx}(x(T)) y(T).
	\end{equation*}

	(iii) For $u(\cdot),v(\cdot) \in \LB$, we have
	\begin{align*}
		& \norm{J''(u(\cdot))v(\cdot)}_{\beta} \\
		\le & \norm{ b_u[\cdot]^\top \psi(\cdot) }_{\beta}
			+ \norm{ H_{ux}[\cdot] y(\cdot) }_{\beta}
			+ \norm{ H_{uu}[\cdot]v(\cdot) }_{\beta} \\
		\le & ( 1 + \norm{u}_1 )
			  \left[
			  	T^\frac{1}{\beta} C_2 ( L + C_1 )\norm{v}_1
			  	+ C_1 \norm{v}_{\beta}
			  \right],
	\end{align*}
	Together with $J''(u(\cdot)) \in \mathbb{L}(\LT)$, this implies that $J''(u(\cdot)) \in \mathbb{L}(\LB)$.
\end{proof}

\begin{remark}\label{RSSC}
	By \cref{PFOD,TSOD}, the stationarity condition \eqref{FNC} and the second-order coercivity condition \eqref{SSC} are equivalent to
	\begin{equation*}
		\left\{
		\begin{aligned}
			& \inner{J'(\bar{u(\cdot)})}{v(\cdot)}_{\LT} = 0, \quad 
			  \forall v(\cdot)\in \LT, \\
			& \inner{J''(\bar{u(\cdot)})v(\cdot)}{v(\cdot)}_{\LT}
			\ge \alpha\norm{v}_2^2, \quad \forall v(\cdot) \in \LT.
		\end{aligned}
		\right.
	\end{equation*}
\end{remark}

\subsection{Continuity of the second derivative}

In general, for any fixed $\beta \in [2,\infty)$, the second-order G{\^a}teaux derivative $J'': \LB \to \mathbb{L}(\LB,\LB)$ of $J$ is not necessarily continuous even at an optimal control. The following example illustrates this.
\begin{example}\label{E-continuous}
	Let $T=1$ and consider the following one-dimensional optimal control problem:
	\begin{equation*}
		\begin{aligned}
			& \underset{u(\cdot)\in \LT}{\mathrm{min}} J(u(\cdot))
			= \frac{1}{2}\int_{0}^{T}\bigl[x(t)^2+\sin^2(u(t))\bigr]\dd t \\
			& \hspace{0.5em} \mathrm{s.t.}
			  \left\{
			  \begin{aligned}
			  	& \dot{x}(t) = u(t), \quad t \in [0,T], \\
			  	& x(0) = 0,
			  \end{aligned}
			  \right.
		\end{aligned}
	\end{equation*}
	This example satisfies \assume and has the solution $\bar{u(\cdot)}=0$. Theorem \ref{TSOD} gives
	\begin{equation*}
		J''(u(\cdot))v(\cdot)
		= \psi(\cdot) + v(\cdot)\cos(2u(\cdot)),
	\end{equation*}
	where, since $b_x=H_{xu}=0$ and $H_{xx}=b_u=1$, equations \eqref{variable} and \eqref{psi} reduce to
	\begin{equation*}
		\dot{y}(t)=v(t), \quad y(0)=0, \qquad
		\dot{\psi}(t)=-y(t), \quad \psi(1)=0.
	\end{equation*}
	Hence, for $t\in[0,1]$,
	\begin{equation*}
		y(t)=\int_0^t v(r)\dd r, \qquad
		\psi(t)=\int_t^1\int_0^s v(r)\dd r\dd s.
	\end{equation*}
	Fix $\beta \in [2,\infty)$. For each integer $k > 1$, define the control
	\begin{equation*}
		u^k(t) =
		\begin{cases}
			\frac{\pi}{4}, & t \in [0,\frac{1}{k}], \\
			0, & t > \frac{1}{k}.
		\end{cases}
	\end{equation*}
	Substituting the expression for $\psi$ into the difference of the second derivatives gives
	\begin{align*}
		\norm{
			J''(u^k(\cdot))-J''(\bar{u(\cdot)})
		}_{\mathbb{L}(\LB)}
		\ge & \frac{
			\norm{
				J''(u^k(\cdot))u^k-J''(\bar{u(\cdot)})u^k
			}_\beta
		}{\norm{u^k}_\beta} \\
		= & \frac{
			\norm{(\cos(2u^k(\cdot))-1)u^k(\cdot)}_\beta
		}{\norm{u^k}_\beta}
		= 1.
	\end{align*}
	Thus, as $k\to \infty$, $\norm{u^k-\bar{u}}_\beta = \frac{\pi}{4k^{\frac{1}{\beta}}} \to 0$, whereas
	\begin{equation*}
		\norm{J''(u^k(\cdot))-J''(\bar{u(\cdot)})}_
		{\mathbb{L}(\LB)}
		\ge 1.
	\end{equation*}
	This shows that $J'':\LB \to \mathbb{L}(\LB)$ is not continuous at the optimal control $\bar{u(\cdot)}$.
\end{example}

The next theorem addresses the continuity of the second-order G{\^a}teaux derivative of $J$.

\begin{theorem}\label{TSODC}
	Suppose that \Assume hold, and let $\beta \in [2,\infty]$.
	\begin{enumerate}
		\item There exists a constant $C_3 > 0$, depending only on $T,L,x_0$, such that, for all $u(\cdot),\bar{u(\cdot)} \in \LI$,
		\begin{equation}\label{SODCI}
			\norm{
				J''(u(\cdot))-J''(\bar{u(\cdot)})
			}_{\mathbb{L}(\LB,\LB)}
			\le C_3(1+\norm{u}_1)\norm{u-\bar{u}}_\infty.
		\end{equation}
		More precisely, $C_3$ is given by
		\begin{equation}\label{C3}
			C_3
			= T C_2 [ L(C_1T + 2) + C_1(T + 2) ] 
			+ C_1 ( T+1 ).
		\end{equation}
		
		\item If $b_{uu},f_{uu}$ are independent of $u$, then there exists a constant $C_4 > 0$, depending only on $T,L,x_0,\beta$, such that, for all $u(\cdot),\bar{u(\cdot)} \in \LB$,
		\begin{equation}\label{SODC}
			\norm{
				J''(u(\cdot))-J''(\bar{u(\cdot)})
			}_{\mathbb{L}(\LB,\LB)}
			\le C_4(1+\norm{u}_1)\norm{u-\bar{u}}_\beta.
		\end{equation}
		More precisely, $C_4$ is given by
		\begin{equation}\label{C4}
			C_4
			= T^{1-\frac{1}{\beta}}
			\Bigl\{
			C_2 [ L(C_1T + 2) + C_1(T + 2) ]
			+ C_1
			\Bigr\}.
		\end{equation}
	\end{enumerate}
\end{theorem}

\begin{proof}
	Let $(x(\cdot),u(\cdot),p(\cdot))$ and $(\bar{x}(\cdot),\bar{u(\cdot)},\bar{p}(\cdot))$ be admissible triples, and let $y(\cdot),\bar{y}(\cdot),\psi(\cdot),\bar{\psi}$ be defined as in \cref{LEypsi}. Theorem \ref{TSOD} and the triangle inequality give
	\begin{equation}\label{SODC.1}
		\begin{aligned}
			& \norm{
				( J''(u(\cdot)) - J''(\bar{u(\cdot)}) ) v(\cdot)
			}_{\beta} \\
			\le & \norm{
				b_u[\cdot]^\top \psi(\cdot)
				- \overline{b_u}[\cdot]^\top \bar{\psi}(\cdot)
			}_{\beta}
			+ \norm{
				H_{ux}[\cdot]y(\cdot)
				- \overline{H_{ux}}[\cdot]
				\bar{y}(\cdot)
			}_{\beta}
			+ \norm{
				( H_{uu}[\cdot]
				- \overline{H_{uu}}[\cdot] )
				v(\cdot)
			}_{\beta},
		\end{aligned}
	\end{equation}
	Denote the three terms on the right-hand side by $\Delta_1,\Delta_2,\Delta_3$, respectively. For $\beta \in [2,\infty]$, \cref{LExpH,LEypsi} give
	\begin{align}
		\notag
		\Delta_1
		\le & \norm{b_u-\overline{b_u}}_{\beta}
		\norm{\psi}_\infty
		+ \norm{\overline{b_u}}_{\beta}
		\norm{\psi-\bar{\psi}}_\infty \\
		\label{SODC.2}
		\le & L C_2 ( 1+\norm{u}_1 ) 
		( T^\frac{1}{\beta} C_1 \norm{u-\bar{u}}_1
		+ \norm{u-\bar{u}}_{\beta} )
		\norm{v}_1
		+ T^\frac{1}{\beta} L C_2 ( 1+\norm{u}_1 )
		\norm{u-\bar{u}}_2 \norm{v}_2 \\
		\notag
		\le & T L C_2 
		( TC_1 + 2 ) ( 1+\norm{u}_1 )
		\norm{u-\bar{u}}_\infty \norm{v}_\beta , \\[0.5em]
		\notag
		\Delta_2
		\le & T^{\frac{1}{\beta}}
		\norm{H_{ux}}_\infty \norm{y-\bar{y}}_\infty
		+ \norm{H_{ux}-\overline{H_{ux}}}_{\beta}
		\norm{\bar{y}}_{\infty} \\ 
		\label{SODC.3}
		\le & C_1 C_2 (1+\norm{u}_1)
		\left[
		T^{\frac{1}{\beta}} \norm{u-\bar{u}}_2 \norm{v}_2
		+ ( T^{\frac{1}{\beta}} \norm{u-\bar{u}}_1
		+ \norm{u-\bar{u}}_{\beta} )
		\norm{v}_1
		\right] \\
		\notag
		\le & T C_1 C_2 
		( T + 2 ) (1+\norm{u}_1)
		\norm{u-\bar{u}}_\infty \norm{v}_\beta , \\[0.5em]
		\notag
		\Delta_3
		\le & \norm{
			H_{uu} - \overline{H_{uu}}
		}_{\infty}
		\norm{v}_{\beta} \\
		\label{SODC.4}
		\le	& C_1 (1+\norm{u}_1)
		( \norm{u-\bar{u}}_1 + \norm{u-\bar{u}}_\infty )
		\norm{v}_\beta \\
		\notag
		\le & C_1 (1+\norm{u}_1) ( T+1 )
		\norm{u-\bar{u}}_\infty \norm{v}_\beta,
	\end{align}
	where
	\begin{equation*}
		\frac{1}{\beta} = 0 \quad \text{when} \quad \beta = \infty.
	\end{equation*}
	With $C_3$ defined by \eqref{C3}, combining \eqref{SODC.1}-\eqref{SODC.4} gives
	\begin{equation}\label{SODC.5}
		\begin{aligned}
			\norm{
				J''(u(\cdot))-J''(\bar{u(\cdot)})
			}_{\mathbb{L}(\LB,\LB)}
			=	& \underset{v(\cdot)\in \LB\setminus\{0\}}{\sup}
			\frac{
				\norm{
					( J''(u(\cdot)) - J''(\bar{u(\cdot)}) )
					v(\cdot)
				}_{\beta}
			}{\norm{v}_\beta} \\
			\le & C_3 (1+\norm{u}_1) \norm{u-\bar{u}}_\infty,
		\end{aligned}
	\end{equation}
	which proves \eqref{SODCI}.
	
	(ii) Arguing as in (i), we obtain
	\begin{align}
		\label{SODC.6}
		\Delta_1
		\le & T^{1-\frac{1}{\beta}} L C_2 
		( TC_1 + 2 ) ( 1+\norm{u}_1 )
		\norm{u-\bar{u}}_\beta \norm{v}_\beta, \\
		\label{SODC.7}
		\Delta_2
		\le & T^{1-\frac{1}{\beta}} C_1 C_2 
		( T + 2 ) (1+\norm{u}_1)
		\norm{u-\bar{u}}_\beta \norm{v}_\beta.
	\end{align}
	Since $b_{uu},f_{uu}$ are independent of $u$, \eqref{EH} and its proof give
	\begin{equation}\label{SODC.8}
		\Delta_3
		\le \norm{
			H_{uu} - \overline{H_{uu}}
		}_{\infty}
		\norm{v}_{\beta}
		\le	C_1 (1+\norm{u}_1) \norm{u-\bar{u}}_1
		\norm{v}_\beta
		\le T^{1-\frac{1}{\beta}}C_1 (1+\norm{u}_1) 
		\norm{u-\bar{u}}_\beta \norm{v}_\beta.
	\end{equation}
	As in \eqref{SODC.5}, with $C_4$ defined by \eqref{C4}, combining \eqref{SODC.1} and \eqref{SODC.6}-\eqref{SODC.8} yields \eqref{SODC}.
\end{proof}



\section{Equivalent characterizations of the second-order coercivity condition}

This section establishes equivalent characterizations of the second-order coercivity condition \eqref{SSC} for use in the convergence analysis of Newton's method.

\subsection{The Riccati equation and uniform convexity}

Let $\SS^k$ denote the linear space of real symmetric $k\times k$ matrices. In this section, we consider the following linear-quadratic optimal control problem:
\begin{equation}\label{LQ}
	\begin{aligned}
		& \underset{v(\cdot)\in  \LT}{\mathrm{min}} \;
		J_L(v(\cdot))
		= \frac{1}{2}\int_{0}^{T}\left\{
		\inner*{  
			\left(\begin{matrix}
				Q(t) & S(t)^\top \\
				S(t) & R(t)
			\end{matrix}\right)
			\left(\begin{matrix}
				y(t) \\
				v(t)
			\end{matrix}\right)}{
			\left(\begin{matrix}
				y(t) \\
				v(t)
			\end{matrix}\right)
		} \right\} \dd t
		+ \frac{1}{2}\inner{Gy(T)}{y(T)} \\
		& \hspace{1em}\mathrm{s.t.}\hspace{1em}\left\{
		\begin{aligned}
			& \dot{y}(t)=A(t)y(t)+B(t)v(t), 
			\quad t\in[0,T], \\
			& y(0) = 0,
		\end{aligned}
		\right.
	\end{aligned}
\end{equation}
where $A(\cdot) \in L^\infty(0,T;\RR^{n\times n})$, $Q(\cdot) \in L^\infty(0,T;\SS^n)$, $B(\cdot) \in L^\infty(0,T;\RR^{n \times m})$, $R(\cdot) \in L^\infty(0,T;\SS^m)$, $G \in \SS^n$, and $S(\cdot) \in L^\infty(0,T;\RR^{m\times n})$. Consider the Riccati equation
\begin{equation}\label{Riccati}
	\left\{
	\begin{aligned}
		& \dot{P}(t)+P(t)A+A^\top P(t)+Q
		- (B^\top P(t)+S)^\top R^{-1}(B^\top P(t)+S)
		=0, \\
		& \qquad \text{a.e. }t\in[0,T], \\
		& P(T) = G,
	\end{aligned}
	\right.
\end{equation}
and let $\varphi:[0,T]\to \RR^n$ satisfy
\begin{equation}\label{varphi}
	\left\{
	\begin{aligned}
		& \dot{\varphi}(t)
		+ [ (A-BR^{-1}S)^\top - PBR^{-1}B^\top ] \varphi(t)
		= 0,
		\quad \text{a.e. }t \in [0,T], \\
		& \varphi(T) = 0.
	\end{aligned}
	\right.
\end{equation}
The following lemma relates problem \eqref{LQ} to the Riccati equation \eqref{Riccati}.
\begin{lemma}\label{L-feedback-solution}
	If there exists a constant $\lambda>0$ such that
	\begin{equation}\label{UPD}
		R(t) \ge \lambda I, \quad \text{a.e. }t\in[0,T],
	\end{equation}
	then the Riccati equation \eqref{Riccati} admits a solution $P(\cdot)$ on the entire interval $[0,T]$ if and only if problem \eqref{LQ} with zero initial state has a unique optimal control. In this case, the optimal control has the feedback representation
	\begin{equation}\label{feedback-solution}
		\bar{v}(t)
		=
		- R^{-1} [(B^\top P(t)+S)\bar{y}(t) + B^\top \varphi(t)],
		\quad \text{a.e. }t \in [0,T].
	\end{equation}
	where $\bar{y}$ is the state corresponding to $\bar{v}$ in problem \eqref{LQ}.
\end{lemma}
\begin{proof}
	The proof follows the argument in \cite[Theorem 2.9 (Chapter 6)]{YongZhou1999}.
\end{proof}

\begin{definition}\label{D-uniformly-convex}
	\begin{enumerate}
		\item For problem \eqref{LQ}, if there exists a constant $\lambda>0$ such that
		\begin{equation}\label{uniformly-convex-1}
			J_L(v(\cdot)) \ge \lambda\norm{v}_2^2, \quad \forall v(\cdot) \in  \LT,
		\end{equation}
		then $J_L$ is said to be uniformly convex in the control for the zero initial state, and $\lambda$ is called a uniform convexity constant.
		\item If condition \eqref{UPD} holds and the Riccati equation \eqref{Riccati} admits a solution, then such a solution is called strongly regular.
	\end{enumerate}
\end{definition}

\subsection{The Riccati equation and second-order coercivity}

We prove that $J_L$ is uniformly convex if and only if the Riccati equation \eqref{Riccati} admits a strongly regular solution, following the approach in \cite{SunLiYong2016}.
\begin{lemma}\label{L-simplify-J_L}
	For any $\Theta(\cdot)\in L^\infty(0,T;\RR^{m\times n})$, if $P(\cdot)$ solves the Lyapunov equation
	\begin{equation}\label{Lyapunov}
		\left\{
		\begin{aligned}
			& \dot{P} + P(A+B\Theta) + (A+B\Theta)^\top P
			+ \Theta^\top R\Theta + S^\top\Theta
			+ \Theta^\top S + Q
			= 0, \quad t\in[0,T], \\
			& P(T)=G,
		\end{aligned}
		\right.
	\end{equation}
	and $Y(\cdot)$ solves the closed-loop system
	\begin{equation}\label{close-loop-system}
		\left\{
		\begin{aligned}
			& \dot{Y}(t)
			= (A(t)+B(t)\Theta(t))Y(t)+B(t)v(t), \quad
			t\in[0,T], \\
			& Y(0)=0,
		\end{aligned}
		\right.
	\end{equation}
	then
	\begin{equation}\label{simplify-J_L}
		\begin{aligned}
			& J_L(\Theta(\cdot) Y(\cdot)+v(\cdot)) \\
			= & \frac{1}{2}\inner{ P(0)y_0}{y_0 } 
			+ \frac{1}{2}\int_{0}^{T}\left\{
			2\inner{ (B^\top P+S+R\Theta)Y}{v } 
			+ \inner{ Rv}{v } 
			\right\}dt.
		\end{aligned}
	\end{equation}
\end{lemma}
\begin{proof}
	Let $z(\cdot)$ be the state corresponding to $\Theta(\cdot)Y(\cdot)+v(\cdot)$ in \eqref{LQ}.
	Setting $Z(\cdot)=z(\cdot)-Y(\cdot)$, we obtain
	\begin{equation*}
		\dot{Z}(t)=A(t)Z(t), \quad t\in(0,T), \quad Z(0)=0,
	\end{equation*}
	Hence
	\begin{equation*}
		z(t)-Y(t)=Z(t)=0, \quad \text{a.e. }t\in[0,T].
	\end{equation*}
	It follows that
	\begin{align*}
		& J_L(\Theta(\cdot)Y(\cdot)+v(\cdot)) \\
		=	& \frac{1}{2}\int_{0}^{T}\left\{
		\inner*{ 
			\left(\begin{matrix}
				Q(t) & S(t)^\top \\
				S(t) & R(t)
			\end{matrix}\right)
			\left(\begin{matrix}
				Y(t) \\
				\Theta(t)Y(t)+v(t)
			\end{matrix}\right)}{
			\left(\begin{matrix}
				Y(t) \\
				\Theta(t)Y(t)+v(t)
			\end{matrix}\right)
		} 
		\right\}dt \\
		& + \frac{1}{2}\inner{ GY(T)}{Y(T) }.
	\end{align*}
	Since $P(T)=G$ and $Y(0)=y_0=0$, the fundamental theorem of calculus yields
	\begin{align*}
		& J_L(\Theta(\cdot) Y(\cdot) + v(\cdot)) \\
		=	& \frac{1}{2} \inner{P(T)Y(T)}{Y(T)}
		+ \frac{1}{2}\int_{0}^{T}
		\inner*{ 
			\left(\begin{matrix}
				Q & S^\top \\
				S & R \\
			\end{matrix}\right)
			\left(\begin{matrix}
				Y \\
				\Theta Y+v \\
			\end{matrix}\right)}{
			\left(\begin{matrix}
				Y \\
				\Theta Y+v \\
			\end{matrix}\right)			
		} dt \\
		=	& \frac{1}{2}\int_{0}^{T}\biggl\{
		\inner*{ 
			\left[
			\dot{P}+P(A+B\Theta)+(A+B\Theta)^\top P
			+ \Theta^\top R\Theta+S^\top\Theta+\Theta^\top S+Q
			\right]Y}{Y} \\
		&	\quad
		+ \inner{Rv}{v} 
		+ 2 \inner{(B^\top P+S+R\Theta)Y}{v}
		\biggr\}dt
		+ \frac{1}{2} \inner{P(0)y_0}{y_0}  \\
		=	& \frac{1}{2} \inner{P(0)y_0}{y_0}
		+ \frac{1}{2} \int_{0}^{T}
		\{
		2 \inner{(B^\top P+S+R\Theta)Y}{v} + \inner{Rv}{v}
		\} \dd t,
	\end{align*}
	which proves \eqref{simplify-J_L}.
\end{proof}
\begin{lemma}\label{L-LQE}
	For every $\Theta(\cdot)\in L^\infty(0,T;\RR^{m\times n})$, there exists a constant $\gamma>0$ such that
	\begin{equation}\label{estimate-LQ}
		\norm{v-\Theta y}_2^2 \ge \gamma \norm{v}_2^2, \quad \forall v(\cdot)\in  \LT
	\end{equation}
	where $y(\cdot)$ solves
	\begin{equation*}
		\left\{
		\begin{aligned}
			& \dot{y}(t) = A(t)y(t)+B(t)v(t), \quad t\in[0,T], \\
			& y(0) = 0.
		\end{aligned}
		\right.
	\end{equation*}
\end{lemma}
\begin{proof}
	Define the operator $\mathcal{T}: \LT\to  \LT$ by
	\begin{equation*}
		\mathcal{T}v=v-\Theta y.
	\end{equation*}
	Then $\mathcal{T}$ is an invertible linear operator, with inverse $\mathcal{T}^{-1}$ given by
	\begin{equation}\label{LQE.1}
		\mathcal{T}^{-1}v=v+\Theta \tilde{y},
	\end{equation}
	where $\tilde{y}$ solves
	\begin{equation*}
		\left\{
		\begin{aligned}
			& \dot{\tilde{y}}(t)
			= (A(t)+B(t)\Theta(t))\tilde{y}(t)
			+ B(t)v(t), \quad t\in[0,T], \\
			& \tilde{y}(0)=0.
		\end{aligned}
		\right.
	\end{equation*}
	Indeed,
	\begin{equation}\label{LQE.2}
		(\mathcal{T}^{-1}\mathcal{T})(v)
		= \mathcal{T}^{-1}(v-\Theta y)
		= v-\Theta y + \Theta \tilde{y}.
	\end{equation}
	By the definitions of $y$ and $\tilde{y}$,
	\begin{equation}\label{LQE.3}
		\left\{
		\begin{aligned}
			& \dot{\tilde{y}}(t)-\dot{y}(t) = (A+B\Theta)(\tilde{y}-y), \\
			& \tilde{y}(0)-y(0)=0.
		\end{aligned}
		\right.
	\end{equation}
	Combining \eqref{LQE.2} and \eqref{LQE.3} therefore gives
	\begin{equation}\label{LQE.4}
		(\mathcal{T}^{-1}\mathcal{T})(v) = v.
	\end{equation}
	Similarly,
	\begin{equation}\label{LQE.5}
		(\mathcal{T} \mathcal{T}^{-1})(v) = v.
	\end{equation}
	Equations \eqref{LQE.4} and \eqref{LQE.5} verify the formula for $\mathcal{T}^{-1}$. In what follows, $\tilde y$ again denotes the solution corresponding to the input $v$. Since $A(\cdot),B(\cdot),\Theta(\cdot)$ are essentially bounded, setting $M=\norm{A+B\Theta}_\infty$ and $N=\norm B_\infty$ gives
	\begin{equation*}
		\abs*{\dot{\tilde{y}}(t)} \le M\abs{\tilde{y}(t)}+N\abs{v(t)}, \quad
		\text{a.e. }t\in[0,T],
	\end{equation*}
	and hence
	\begin{equation}\label{LQE.6}
		\abs{\tilde{y}(t)}
		\le \int_{0}^{t}(M\abs{\tilde{y}(s)}+N\abs{v(s)})ds
		\le \int_{0}^{t}M\abs{\tilde{y}(s)}ds+N\norm{v}_1, 
		\quad \forall t\in[0,T].
	\end{equation}
	Applying Gronwall's inequality to \eqref{LQE.6}, we obtain
	\begin{equation*}
		\abs{\tilde{y}(t)}
		\le Ne^{Mt}\norm{v}_1
		\le \sqrt{T}Ne^{MT}\norm{v}_2,
	\end{equation*}
	Together with H{\"o}lder's inequality, this gives
	\begin{equation}\label{LQE.7}
		\norm{\Theta\tilde{y}}_2
		\le \sqrt{T}\norm{\Theta}_\infty \norm{\tilde{y}}_\infty
		\le TNe^{MT}\norm{\Theta}_\infty\norm{v}_2.
	\end{equation}
	Combining \eqref{LQE.1} and \eqref{LQE.7}, we obtain
	\begin{equation*}
		\norm{\mathcal{T}^{-1}v}_2
		\le \norm{v}_2+\norm{\Theta \tilde{y}}_2
		\le (TNe^{MT}\norm{\Theta}_\infty+1)\norm{v}_2,
	\end{equation*}
	Thus, $\mathcal{T}^{-1}:\LT \to \LT$ is a bounded linear operator, and consequently
	\begin{equation*}
		\norm{ v }_2^2
		= \norm{ \mathcal{T}^{-1}\mathcal{T}v }_2^2
		\le \norm{ \mathcal{T}^{-1}}_{\mathbb{L}( \LT)}^2
		\norm{ \mathcal{T}v }_2^2
	\end{equation*}
	Taking $\gamma=\norm{\mathcal{T}^{-1}}_{\mathbb{L}( \LT)}^{-2}$ yields \eqref{estimate-LQ}.
\end{proof}

\begin{proposition}\label{P-convex-positive}
	If $\lambda > 0$ is a uniform convexity constant of $J_L$, then \eqref{UPD} holds with the constant $2\lambda$.
\end{proposition}
\begin{proof}
	For any $\Theta\in L^\infty(0,T;\RR^{m\times n})$, let $P(\cdot)$ solve the Lyapunov equation \eqref{Lyapunov}. For any $v\in  \LT$, let $Y(\cdot)$ solve the closed-loop system \eqref{close-loop-system}. By the uniform convexity of $J_L$ and Lemma \ref{L-simplify-J_L},
	\begin{align*}
		\lambda \norm{\Theta Y+v}_2^2
		\le & J_L(\Theta(\cdot)Y(\cdot)+v(\cdot)) \\
		=	& \frac{1}{2}\int_{0}^{T}\left\{2\inner{  (B^\top P+S+R\Theta)Y}{v} +\inner{  Rv}{v} \right\}dt.
	\end{align*}
	Thus, for every $v(\cdot)\in  \LT$,
	\begin{equation}\label{proof-R-1}
		\int_{0}^{T}\left\{
		2\inner{  (B^\top P+S+(R-2\lambda I)\Theta)Y}{v} 
		+ \inner{  (R-2\lambda I)v}{v} 
		\right\}dt
		\ge 2 \lambda \norm{\Theta Y}_2^2
		\ge 0.
	\end{equation}
	Fix $v_0\in\RR^m$ and set
	\begin{equation*}
		v(t)=v_0I_{[s,s+h]}(t), \quad 0\le s<s+h\le T,
	\end{equation*}
	where $I_{[s,s+h]}$ denotes the indicator function of $[s,s+h]$. Then
	\begin{equation*}
		\left\{
		\begin{aligned}
			& \dot Y_0(t)
			= (A(t)+B(t)\Theta(t))Y_0(t)
			+ B(t)v_0I_{[s,s+h]}(t), \quad t\in[0,T], \\
			& Y_0(0)=0.
		\end{aligned}
		\right.
	\end{equation*}
	Hence
	\begin{equation*}
		Y_0(t)=
		\begin{cases}
			0, & t\in[0,s], \\
			\Phi(t)\int_{s}^{t\wedge (s+h)}\Phi(r)^{-1}B(r)v_0dr, & t\in[s,T],
		\end{cases}
	\end{equation*}
	where $\Phi(\cdot)$ is the fundamental matrix satisfying
	\begin{equation*}
		\left\{
		\begin{aligned}
			& \dot{\Phi}(t)=(A(t)+B(t)\Theta(t))\Phi(t), \quad
			t\in[0,T], \\
			& \Phi(0)=I.
		\end{aligned}
		\right.
	\end{equation*}
	Thus, \eqref{proof-R-1} becomes
	\begin{align*}
		0
		\le &
		~ \frac{1}{2}\int_{s}^{s+h}\biggl\{
		2\inner*{  
		\left[B^\top P+S+(R-2\lambda I)\Theta\right]\Phi(t)
		\int_{s}^{t}\Phi(r)^{-1}B(r)v_0dr}{ v_0
		} \\
		& + \inner{  (R-2\lambda I)v_0}{v_0} 
		\biggr\}dt.
	\end{align*}
	The integral of the first term in this inequality is $O(h^2)$. Dividing both sides by $h$ and letting $h\to 0^+$ at a Lebesgue point of $R$, we obtain
	\begin{equation*}
		\inner{ (R(t)-2\lambda I)v_0}{v_0} \ge 0, \quad
		\text{a.e. }t\in[0,T], ~ \forall v_0\in\RR^m.
	\end{equation*}
	First take a common set of full measure for all $v_0\in\QQ^m$, and then extend the inequality to all $v_0\in\RR^m$ by continuity in $v_0$. Thus, \eqref{UPD} holds with the constant $2\lambda>0$.
\end{proof}

\begin{proposition}\label{P-equivalence}
	The following statements are equivalent:
	\begin{enumerate}
		\item $J_L$ is uniformly convex;
		\item The Riccati equation \eqref{Riccati} admits a strongly regular solution.
	\end{enumerate}
\end{proposition}
\begin{proof}
	(i) $\Rightarrow$ (ii) Suppose that $J_L$ is uniformly convex for the zero initial state. By Proposition \ref{P-convex-positive}, there exists a constant $\lambda>0$ such that \eqref{UPD} holds. Moreover, since $J_L$ is uniformly convex and $J_L(0)=0$, the zero control is the unique optimal control of problem \eqref{LQ}. Lemma \ref{L-feedback-solution} therefore implies that the Riccati equation \eqref{Riccati} admits a strongly regular solution.
	
	(ii) $\Rightarrow$ (i)
	Suppose that \eqref{UPD} holds and that $P(\cdot)$ solves the Riccati equation \eqref{Riccati}. Set $\Theta=-R^{-1}(B^\top P+S)\in L^\infty(0,T;\RR^{m\times n})$. For any $v(\cdot)\in  \LT$, let $y(\cdot)$ solve
	\begin{equation*}
		\dot{y}(t)=A(t)y(t)+B(t)v(t), ~t\in[0,T], \quad
		y(0)=0.
	\end{equation*}
	Since $P(T)=G$ and $y(0) = 0$, the fundamental theorem of calculus gives
	\begin{align*}
		J_L(v(\cdot))
		=	& \frac{1}{2} \inner{P(T)y(T)}{y(T)}
		+ \frac{1}{2}\int_{0}^{T}\left\{
		\inner{Qy}{y}
		+ 2 \inner{Sy}{v}
		+ \inner{Rv}{v} \right\} \dd t \\
		=	& \frac{1}{2}\int_{0}^{T}\biggl\{
		\inner{(\dot{P}+PA+A^\top P+Q)y}{y}
		+ 2 \inner{(B^\top P+S)y}{v} \\
		&		+ \inner{Rv}{v}
		\biggr\}dt
		+ \frac{1}{2} \inner{P(0)y(0)}{y(0)} \\
		=	& \frac{1}{2}\int_{0}^{T}\left\{
		\inner{\Theta^\top R\Theta y}{y}
		- 2\inner{R\Theta y}{v}  + \inner{Rv}{v}
		\right\}dt \\
		=	& \frac{1}{2}\int_{0}^{T}
		\inner{R(v-\Theta y)}{v-\Theta y}
		\dd t \\
		\ge & \frac{\lambda}{2} \norm{v-\Theta y}_2^2
		\ge	  \frac{\lambda\gamma}{2} \norm{v}_2^2.
	\end{align*}
	This proves that $J_L$ is uniformly convex.
\end{proof}

Let $(\bar{x}(\cdot),\bar{u(\cdot)},\bar p(\cdot))$ be an admissible triple, and set
\begin{align}
	\notag
	& A(\cdot)
	= b_x(\cdot,\bar{x}(\cdot),\bar{u(\cdot)}),~
	B(\cdot)
	= b_u(\cdot,\bar{x}(\cdot),\bar{u(\cdot)}), \\
	\label{symbol-bar}
	& Q(\cdot)
	= H_{xx}(
	\cdot,\bar{x}(\cdot),\bar{u(\cdot)},\bar{p}(\cdot)
	),~
	S(\cdot)
	= H_{ux}(
	\cdot,\bar{x}(\cdot),\bar{u(\cdot)},\bar{p}(\cdot)
	), \\
	\notag
	& R(\cdot)
	= H_{uu}(
	\cdot,\bar{x}(\cdot),\bar{u(\cdot)},\bar{p}(\cdot)
	),~
	G
	= g_{xx}(\bar{x}(T)),
\end{align}
Proposition \ref{P-equivalence} then yields an equivalent characterization of the second-order coercivity condition \eqref{SSC}.

\begin{theorem}\label{TSSCE}
	Suppose that \Assume hold, and let $\bar{u(\cdot)}$ be an optimal control of problem \eqref{DOCP}. Then the second-order coercivity condition \eqref{SSC} holds for some positive coercivity constant if and only if the Riccati equation \eqref{Riccati}, with coefficients defined by \eqref{symbol-bar}, admits a strongly regular solution. Furthermore, if \eqref{SSC} holds with a constant $\alpha > 0$, then $R \ge \alpha I$.
\end{theorem}

\begin{proof}
	By \cref{LExpH}, the coefficients $A(\cdot),B(\cdot),Q(\cdot),S(\cdot),R(\cdot),G$ defined in \eqref{symbol-bar} satisfy the requirements of problem \eqref{LQ}. Substituting \eqref{symbol-bar} into $J_L$ gives
	\begin{align*}
		J_L(v(\cdot))
		=	& \frac{1}{2}\int_{0}^{T}
		H_{(x,u)^2}(t,\bar{x}(t),\bar{u}(t),\bar{p}(t))
		(y(t),v(t))^2 dt
		+ \frac{1}{2}y(T)^\top g_{xx}(\bar{x}(T))y(T).
	\end{align*}
	Thus, \eqref{SSC} holds with $\alpha$ if and only if $\frac{\alpha}{2}$ is a uniform convexity constant of $J_L$. By \cref{P-convex-positive,P-equivalence}, the existence of such an $\alpha>0$ is equivalent to the existence of a strongly regular solution of \eqref{Riccati}; moreover, $R \ge \alpha I$ for every $\alpha$ satisfying \eqref{SSC}.
\end{proof}

\begin{remark}
	The second-order coercivity condition \eqref{SSC} is an abstract functional inequality on an infinite-dimensional space. Since it must hold for every function in $\LT$, it does not lend itself to direct verification in numerical implementations. The equivalence in \cref{TSSCE} reduces this verification to the existence of a strongly regular solution of the associated differential Riccati equation \eqref{Riccati}. In engineering applications, \cref{TSSCE} leads to the following two checks:
	\begin{enumerate}
		\item Uniform positive definiteness: this requires checking the positive definiteness of the finite-dimensional matrix $R(\cdot)$ at each time and ensuring a positive lower bound uniform in time.
		\item Numerical solution of an ordinary differential equation: the Riccati equation \eqref{Riccati} can be solved numerically. If the numerical solution exists throughout the interval without encountering a singularity, this provides numerical evidence for the solvability of the Riccati equation.
	\end{enumerate}
	In practice, these two checks are more direct and feasible than testing the original integral inequality over an infinite-dimensional space. For solvability conditions for Riccati equations, see \cite{Kalman1960,Anderson2007,Kuvcera1973}; for related numerical methods, see \cite{Dieci1994,Choi1989,Kleinman1968,Laub1979,Lang2015,Blanes2024}.
\end{remark}


\section{Convergence analysis of Newton's method}

Using the second-derivative estimates in Section~3 and the characterization of second-order coercivity in Section~4, we establish local quadratic convergence of Newton's method and derive explicit estimates for the radii of convergence. Numerical examples are given to illustrate the results.

\subsection{Newton iteration}

\begingroup
\interdisplaylinepenalty=10000
For problem~\eqref{DOCP}, let $u^k(\cdot)$ be the current iterate, and let $x^k(\cdot)$ and $p^k(\cdot)$ be the corresponding state and adjoint variables, respectively. Consider the following subproblem $\mathrm{(DOCP)^k}$:
\begin{align}
	\notag
	& \underset{v(\cdot)\in  \LT}{\mathrm{min}}~
	J^k(v(\cdot))
	= \frac{1}{2}\int_{0}^{T}\biggl\{
	2 \inner{H_u^k[t]}{v(t)}
	+ H_{(x,u)^2}^k[t]
	(y(t),v(t))^2
	\biggr\}\dd t \\
	\label{DOCPk}
	& \hspace{8em}
	+ \frac{1}{2} \inner{g_{xx}(x^k(T))y(T)}{y(T)} \\
	\notag
	& \hspace{1em}\mathrm{s.t.} \hspace{1em}
	\left\{\begin{aligned}
		& \dot{y}(t)
		= b_x(t,x^k(t),u^k(t))y(t)
		+ b_u(t,x^k(t),u^k(t))v(t), 
		~ t \in [0,T], \\
		& y(0) = 0.
	\end{aligned}\right.
\end{align}
\endgroup

\begin{algorithm}[H] 
	\caption{Newton's method}
	\label{Algorithm-DOCP-Newton}
	\begin{algorithmic}[1]
		\item[\bfseries Initialization.]
		Choose an initial control $u^0(\cdot)$.
		\item[\bfseries Step~1.]
		Given the current control $u^k(\cdot)$, compute the corresponding state $x^k(\cdot)$ and adjoint $p^k(\cdot)$ from \eqref{system} and \eqref{adjoint}.
		\item[\bfseries Step~2.]
		Compute the derivative of the cost functional, $J'(u^k(\cdot)) = H_u^k[\cdot]$. If $J'(u^k(\cdot))=0$, stop and return the stationary control $u^k(\cdot)$; otherwise, proceed to Step~3.
		\item[\bfseries Step~3.]
		Solve problem~\eqref{DOCPk} to obtain $v^{k}(\cdot)$. If $v^k(\cdot)=0$, stop and return the stationary control $u^k(\cdot)$; otherwise, update $u^{k+1}(\cdot)=u^k(\cdot)+v^{k}(\cdot)$ and return to Step~1.
	\end{algorithmic}
\end{algorithm}

\begin{remark}\label{RSQP}
	If the SQP method is used instead, subproblem~\eqref{DOCPk} is replaced by
	\begin{equation}\label{SQPk}
		\begin{aligned}
			& \underset{
				\underset{y(\cdot)\in W^{1,2}(0,T;\RR^n)}
				{v(\cdot)\in \LT}
			}{\min} \;
			\frac{1}{2}\int_{0}^{T}\biggl\{
			2[
			f_u^k[t]^\top v(t)
			+ f_x^k[t]^\top y(t)
			]
			+ H_{(x,u)^2}^k[t]
			(y(t),v(t))^2
			\biggr\} \dd t \\
			& \hspace{6 em}
			+ g_x(x^k(T))^\top y(T)
			+ \frac{1}{2}y(T)^\top g_{xx}(x^k(T))y(T) \\
			& \hspace{2em}\mathrm{s.t.} \hspace{2.5em}
			\left\{
			\begin{aligned}
				& \dot{y}(t)
				= b(t,x^k(t),u^k(t))
				+ b_x(t,x^k(t),u^k(t))y(t) \\
				& \hspace{2em}
				+ b_u(t,x^k(t),u^k(t))v(t)
				- \dot{x^k}(t), ~
				~ t\in[0,T], \\
				& y(0) = 0,
			\end{aligned}
			\right.
		\end{aligned}
	\end{equation}
	At each iteration, solve \eqref{SQPk} for $(y^k(\cdot),v^k(\cdot))$ and update
	\begin{equation*}
		(x^{k+1}(\cdot),u^{k+1}(\cdot))
		= (x^k(\cdot),u^k(\cdot))
		+ (y^k(\cdot),v^{k}(\cdot)).
	\end{equation*}
	Let $\{(y_k(\cdot),v_k(\cdot))\}$ denote the sequence of subproblem solutions generated by the SQP method, and let $(x^k(\cdot),u^k(\cdot))$ and $(x_k(\cdot),u_k(\cdot))$ denote the state--control pairs generated by Newton's method and the SQP method, respectively. Then
	\begin{equation*}
		\left\{
		\begin{aligned}
			& \dot{x_{k+1}}(t) = b(t,x_k(t),u_k(t))
			+ b_x(t,x_k(t),u_k(t))y_k(t) \\
			& \qquad\qquad
			+ b_u(t,x_k(t),u_k(t))v_k(t),
			\quad t\in[0,T], \\
			& x^{k+1}(0) = x_0, ~ x_{k+1}(0) = x_0.
		\end{aligned}
		\right.
	\end{equation*}
	When the control system is nonlinear, the state--control pairs generated by the SQP method satisfy the linearized system but generally do not satisfy the original nonlinear system. In contrast, the state--control pairs generated by Newton's method satisfy the original nonlinear system exactly.
\end{remark}

\subsection{Local quadratic convergence of Newton's method}

Consider the linear-quadratic optimal control problem with a linear term in the cost functional
\begin{equation}\label{LQL}
	\begin{aligned}
		& \underset{v(\cdot)\in  \LT}{\mathrm{min}} \;
		\frac{1}{2}\int_{0}^{T}\left\{
		2 \inner{\rho(t)}{v(t)}
		+ \inner*{  
			\left(\begin{matrix}
				Q(t) & S(t)^\top \\
				S(t) & R(t)
			\end{matrix}\right)
			\left(\begin{matrix}
				y(t) \\
				v(t)
			\end{matrix}\right)}{
			\left(\begin{matrix}
				y(t) \\
				v(t)
			\end{matrix}\right)
		} \right\}\dd t
		+ \frac{1}{2}\inner{Gy(T)}{y(T)} \\
		& \hspace{1em}\mathrm{s.t.}\hspace{1em}\left\{
		\begin{aligned}
			& \dot{y}(t)=A(t)y(t)+B(t)v(t), 
			\quad t\in[0,T], \\
			& y(0) = 0,
		\end{aligned}
		\right.
	\end{aligned}
\end{equation}
where $\rho \in \LT$ and the remaining coefficients are as in \eqref{LQ}. The following lemma holds.

\begin{lemma}\label{LFS}
	Suppose that \eqref{UPD} holds. If problem~\eqref{LQL} has a unique optimal pair $(\bar{y}(\cdot),\bar{v}(\cdot))$, then there exists a function $h(\cdot)$ satisfying
	\begin{equation}\label{h}
		\dot{h}(t)
		=
		- A(t)^\top h(t)
		- Q(t) \bar{y}(t)
		- S(t)^\top \bar{v}(t),
		\quad t \in[0,T], \quad h(T) = G\bar{y}(T),
	\end{equation}
	and
	\begin{equation}\label{feedback-solution-yh}
		\bar{v}(t)
		= - R(t)^{-1} [
		B(t)^\top h(t) + S(t) \bar{y}(t) + \rho(t)
		], \quad t \in [0,T].
	\end{equation}
\end{lemma}
\begin{proof}
	The integral term in the cost functional of \eqref{LQL} differs from that of the linear-quadratic optimal control problem considered in \cite[Theorem 2.3 (Chapter 6)]{YongZhou1999} only by the additional linear term $\inner{\rho(t)}{v(t)}$. The argument used in the proof of \cite[Theorem 2.3 (Chapter 6)]{YongZhou1999} also applies here and proves Lemma~\ref{LFS}.
\end{proof}

\begin{theorem}[Invertibility of the second derivative]\label{TSODP}
	\Assume hold. Let $\bar{u(\cdot)} \in \LT$, and suppose that \eqref{SSC} holds at this point.
	\begin{enumerate}
		\item The operator $J''(\bar{u(\cdot)}) \in \mathbb{L}(\LT)$ is invertible, and its inverse $J''(\bar{u(\cdot)})^{-1} \in \mathbb{L}(\LT)$ satisfies
		\begin{equation}\label{SODPT-bar}
			\norm{
				J''(\bar{u(\cdot)})^{-1}
			}_{\mathbb{L}(\LT)}
			\le \frac{1}{\alpha}.
		\end{equation}
		If $b_{uu},f_{uu}$ are independent of $u$, then for every $\eta \in (0,\alpha)$ there exists a constant $r_1 > 0$ such that, for every $u(\cdot) \in B_{\LT}(\bar{u},r_1)$, the operator $J''(u(\cdot)) \in \mathbb{L}(\LT)$ is invertible and satisfies
		\begin{equation}\label{SODPT}
			\norm{
				J''(u(\cdot))^{-1}
			}_{\mathbb{L}(\LT)}
			\le \frac{1}{\eta}.
		\end{equation}
		Moreover, with $C_4$ given by \eqref{C4} for $\beta=2$, one may take
		\begin{equation}\label{r1}
			r_1 = \frac{\alpha-\eta}{C_4(1+\norm{\bar{u}}_1)}.
		\end{equation}
		
		\item If $\bar{u(\cdot)} \in \LI$, then there exists a constant $\gamma > 0$ such that, for every $\eta \in (0,\gamma)$, there exists a constant $r_2 > 0$ satisfying
		\begin{equation}\label{SODPI}
			\norm{
				J''(u(\cdot))^{-1}
			}_{\mathbb{L}(\LI)}
			\le \frac{1}{\eta},
			\quad u(\cdot) \in B_{\LI}(\bar{u},r_2)
		\end{equation}
		Moreover, $\gamma$ and $r_2$ are given by
		\begin{equation}\label{gamma-r2}
			\left\{\begin{aligned}
				& C_h
				= C_1 e^{LT} (TLe^{LT}+1) + L^2e^{2LT}, \\
				& \gamma
				= \frac{\alpha^2}{
					T ( LC_h + C_1Le^{LT} )
					(1+\norm{\bar{u}}_1)
					+ \alpha
				}, \\
				& r_2
				= \frac{
					\gamma - \eta
				}{C_3 (1+\norm{\bar{u}}_1)}.
			\end{aligned}\right.
		\end{equation}
	\end{enumerate}
\end{theorem}

\begin{proof}
	(i) Suppose that $J''(\bar{u(\cdot)})$ is not injective. Then there exist distinct $v_1(\cdot),v_2(\cdot)\in \LT$ for which \eqref{SSC} gives
	\begin{equation*}
		0
		= \inner{J''(\bar{u(\cdot)})v(\cdot)}{v(\cdot)}_{\LT}
		\ge \alpha \norm{v}_2^2 > 0, \quad v(\cdot)=v_1(\cdot)-v_2(\cdot),
	\end{equation*}
	This is a contradiction. Hence $J''(\bar{u(\cdot)})$ is injective. For any Cauchy sequence $\{J''(\bar{u(\cdot)})v_n(\cdot)\}$ in $\LT$, condition~\eqref{SSC} yields
	\begin{equation*}
		\norm{v_m-v_n}_2 
		\le \frac{1}{\alpha}\norm{
			J''(\bar{u(\cdot)})
			( v_m(\cdot)-v_n(\cdot) )
		}_2, \quad \forall n,m\in \NN.
	\end{equation*}
	Thus $\{v_n(\cdot)\}$ is also a Cauchy sequence in $\LT$. Since $\LT$ is a Banach space, there exists $v(\cdot) \in \LT$ such that
	\begin{equation*}
		\underset{n \to \infty}{\lim} v_n(\cdot) = v(\cdot).
	\end{equation*}
	The boundedness of $J''(\bar{u(\cdot)})$ then implies
	\begin{equation*}
		\underset{n\to\infty}{\lim}J''(\bar{u(\cdot)})v_n(\cdot)
		= J''(\bar{u(\cdot)})v(\cdot).
	\end{equation*}
	Therefore, the range $\mathrm{Im}(J''(\bar{u(\cdot)}))$ of $J''(\bar{u(\cdot)})$ is closed in $\LT$. Suppose that $\mathrm{Im}(J''(\bar{u(\cdot)})) \ne \LT$. Since $\mathrm{Im}(J''(\bar{u(\cdot)}))$ is closed and $\LT$ is a Hilbert space, the orthogonal decomposition of $\LT$ gives a nonzero $w(\cdot)\in \LT$ such that
	\begin{equation*}
		\inner{	J''(\bar{u(\cdot)})v(\cdot)}{w(\cdot)}_{\LT}
		= 0, \quad \forall v(\cdot) \in \LT.
	\end{equation*}
	In particular, taking $v(\cdot) = w(\cdot)$ gives
	\begin{equation*}
		\inner{	J''(\bar{u(\cdot)})w(\cdot)}{w(\cdot)}_{\LT} = 0,
	\end{equation*}
	This contradicts \eqref{SSC}. Hence $\mathrm{Im}(J''(\bar{u(\cdot)}))=\LT$, and $J''(\bar{u(\cdot)})$ is bijective.
	
	Consequently, the inverse $J''(\bar{u(\cdot)})^{-1}:\LT \to \LT$ of $J''(\bar{u(\cdot)})$ exists. For every $v(\cdot)\in \LT$, condition~\eqref{SSC} gives
	\begin{equation*}
		\alpha \norm{v}_2^2 
		\le \inner{J''(\bar{u(\cdot)})v(\cdot)}{v(\cdot)}_{\LT}
		\le \norm{J''(\bar{u(\cdot)})v(\cdot)}_2
		\norm{v}_2,
	\end{equation*}
	that is,
	\begin{equation}\label{SODPT.1}
		\norm{v}_2
		\le \frac{1}{\alpha} \norm{J''(\bar{u(\cdot)})v(\cdot)}_2,
		\quad \forall v(\cdot) \in \LT.
	\end{equation}
	It follows that $J''(\bar{u(\cdot)})^{-1}$ satisfies
	\begin{equation*}
		\norm{ J''(\bar{u(\cdot)})^{-1} }_{\mathbb{L}(\LT)}
		= \underset{v(\cdot)\in \LT\setminus\{0\}}{\sup}\frac{\norm{v}_2}
		{\norm{J''(\bar{u(\cdot)})v(\cdot)}_2}
		\le \frac{1}{\alpha}.
	\end{equation*}
	
	When $b_{uu},f_{uu}$ are independent of $u$, interchange $u,\bar u$ in Theorem~\ref{TSODC}(ii) and take $\beta=2$. H{\"o}lder's inequality then gives, for every $v(\cdot) \in \LT$,
	\begin{equation}\label{SODPT.2}
		\begin{aligned}
			\abs*{
				\inner{
					( J''(u(\cdot))- J''(\bar{u(\cdot)}) ) v(\cdot)
				}{v(\cdot)}_{\LT}
			} 
			\le & \norm{
				J''(u(\cdot))- J''(\bar{u(\cdot)})
			}_{\mathbb{L}(\LT)} \norm{v}_2^2 \\
			\le & C_4 (1+\norm{\bar{u}}_1) \norm{u-\bar{u}}_2 \norm{v}_2^2
		\end{aligned}
	\end{equation}
	For any $\eta \in (0,\alpha)$, let $r_1$ be defined by \eqref{r1}. If $u(\cdot)\in B_{\LT}(\bar{u},r_1)$, then \eqref{SODPT.2} implies
	\begin{equation}\label{SODPT.3}
		\begin{aligned}
			& J''(u(\cdot))(v(\cdot),v(\cdot)) \\
			\ge & \inner{
				J''(\bar{u(\cdot)})v(\cdot)
			}{v(\cdot)}_{\LT}
			- \abs*{
				\inner{
					( J''(u(\cdot))- J''(\bar{u(\cdot)}) ) v(\cdot)
				}{v(\cdot)}_{\LT}
			} \\
			\ge & \eta\norm{v}_2^2, \quad \forall v(\cdot)\in \LT.
		\end{aligned}
	\end{equation}
	The same argument as in the proof of \eqref{SODPT-bar} establishes \eqref{SODPT}.
	
	(ii) Using the notation in \eqref{symbol-bar}, consider the following perturbed linear-quadratic optimal control problem:
	\begin{equation}\label{SODPI.1}
		\begin{aligned}
			& \underset{v(\cdot)\in  \LT}{\mathrm{min}}~
			\frac{1}{2}\int_{0}^{T}\Biggl\{
			\inner*{ 
				\left(\begin{matrix}
					Q(t) & S(t)^\top \\
					S(t) & R(t)
				\end{matrix}\right)
				\left(\begin{matrix}
					y(t) \\
					v(t)
				\end{matrix}\right)}{
				\left(\begin{matrix}
					y(t) \\
					v(t)
				\end{matrix}\right)
			} 
			+ 2 \inner{\eps(t)}{v(t)}
			\Biggr\} \dd t
			+ \frac{1}{2} \inner{Gy(T)}{y(T)} \\
			& \hspace{1em}\mathrm{s.t.}\;
			\left\{\begin{aligned}
				& \dot{y}(t)=A(t)y(t)+B(t)v(t), 
				\quad t\in[0,T], \\
				& y(0) = 0,
			\end{aligned}\right.
		\end{aligned}
	\end{equation}
	where $\eps(\cdot) \in \LI$ is the perturbation. By \cref{TSOD}(ii), problem~\eqref{SODPI.1} is equivalent to
	\begin{equation}\label{SODPI.1-simple}
		\underset{v(\cdot)\in \LT}{\mathrm{min}} \;
		\inner{\eps(\cdot)}{v(\cdot)}_{\LT}
		+ \frac{1}{2}
		\inner{
			J''(\bar{u(\cdot)})v(\cdot)
		}{v(\cdot)}_{\LT}.
	\end{equation}
	By (i), the operator $J''(\bar{u(\cdot)}) \in \mathbb{L}(\LT)$ is invertible with a bounded inverse. Hence \eqref{SODPI.1} has a unique optimal pair $(\bar{y}^\eps(\cdot),\bar{v}^\eps(\cdot))$ satisfying
	\begin{equation}\label{SODPI.2}
		\bar{v}^\eps(\cdot)
		= - J''(\bar{u(\cdot)})^{-1} \eps(\cdot).
	\end{equation}
	By \eqref{SSC}, $\alpha/2$ is a uniform convexity constant for the corresponding $J_L$. Thus \cref{P-convex-positive} yields $R(t)\ge\alpha I$ a.e. By \cref{LFS}, there exists a function $h^\eps(\cdot)$ satisfying
	\begin{equation*}
		\left\{\begin{aligned}
			& \dot{h^\eps}(t)
			= - A(t)^\top h^\eps(t)
			- Q(t)\bar{y}^\eps(t)
			- S(t)^\top \bar{v}^\eps(t),
			\quad t \in [0,T], \\
			& h^\eps(T) = G\bar{y}^\eps(T)
		\end{aligned}\right.
	\end{equation*}
	and the optimal control admits the representation
	\begin{equation}\label{SODPI.4}
		\bar{v}^\eps(t)
		= - R(t)^{-1}
		[
		B(t)^\top h^\eps(t)
		+ S(t)\bar{y}^\eps(t)
		+ \eps(t)
		], \quad t\in[0,T].
	\end{equation}
	Thus the optimal control $\bar{v}^\eps(\cdot)$ of problem~\eqref{SODPI.1} belongs to $\LI$. By (i) and \eqref{SODPI.1},
	\begin{equation}\label{SODPI.3}
		\norm{\bar{v}^\eps}_2
		\le \frac{1}{\alpha} \norm{\eps}_2
		\le \frac{\sqrt{T}}{\alpha} \norm{\eps}_\infty.
	\end{equation}
	Equation~\eqref{SODPI.1} and Gronwall's inequality give
	\begin{equation}\label{SODPI.5}
		\underset{t\in[0,T]}{\sup}
		\abs{\bar{y}^\eps(t)}
		\le Le^{LT}\norm{\bar{v}^\eps}_1.
	\end{equation}
	By \eqref{h}, \eqref{SODPI.4}, and Gronwall's inequality,
	\begin{align*}
		\underset{t\in[0,T]}{\sup}
		\abs{h^\eps(t)}
		\le & e^{\int_{0}^{T}\abs{A(t)}\dd t}
		\Bigl[
		\abs{G\bar{y}^\eps(T)}
		+ \norm{Q\bar{y}^\eps}_1
		+ \norm{S^\top\bar{v}^\eps}_1
		\Bigr] \\
		\le & e^{LT}
		\Bigl[
		C_1(1+\norm{\bar{u}}_1)(TLe^{LT}+1)
		+ L^2e^{LT}
		\Bigr] \norm{\bar{v}^\eps}_1.
	\end{align*}
	Setting $	C_h = C_1 e^{LT} (TLe^{LT}+1) + L^2e^{2LT} $,
	we obtain
	\begin{equation}\label{SODPI.6}
		\underset{t\in[0,T]}{\sup}
		\abs{h^\eps(t)}
		\le C_h (1+\norm{\bar{u}}_1)
		\norm{\bar{v}^\eps}_1.
	\end{equation}
	Combining \eqref{SODPI.2}--\eqref{SODPI.6} and applying \cref{A1}, \cref{LExpH}, and \cref{P-equivalence}, we obtain
	\begin{equation}
		\begin{aligned}
			\abs{\bar{v}^\eps(t)}
			\le & \abs{R(t)^{-1}}
			\Bigl[
			\abs{B(t)}\abs{h^\eps(t)}
			+ \abs{S(t)}\abs{\bar{y}^\eps(t)}
			+ \abs{\eps(t)}
			\Bigr] \\
			\le & \frac{\norm{\bar{v}^\eps}_1}{\alpha}
			( LC_h + C_1Le^{LT} ) (1+\norm{\bar{u}}_1)
			+ \frac{\abs{\eps(t)}}{\alpha} \\
			\le & \frac{T\norm{\eps}_\infty}{\alpha^2}
			( LC_h + C_1Le^{LT} ) (1+\norm{\bar{u}}_1)
			+ \frac{\abs{\eps(t)}}{\alpha}.
		\end{aligned}
	\end{equation}
	With $\gamma$ defined by \eqref{gamma-r2}, combining \eqref{SODPI.1} and \eqref{SODPI.6} therefore yields
	\begin{equation}\label{SODPI.7}
		\norm{
			J''(\bar{u(\cdot)})^{-1}
		}_{\mathbb{L}(\LI)}
		= \underset{\eps(\cdot) \in \LI\setminus\{0\}}{\sup}
		\frac{\norm{\bar{v}^\eps}_\infty}{\norm{\eps}_\infty}
		\le \frac{1}{\gamma}
	\end{equation}
	For a given $\eta \in (0,\gamma)$, define $r_2$ as in \eqref{gamma-r2}. Interchanging $u,\bar u$ in \cref{TSODC}(i) gives
	\begin{equation}\label{SODPI.8}
		\left\{\begin{aligned}
			& \norm{ J''(u(\cdot)) - J''(\bar{u(\cdot)}) }_{\mathbb{L}(\LI)}
			\le C_3 (1+\norm{\bar{u}}_1) \norm{u-\bar{u}}_\infty
			\le \gamma - \eta, \\
			& \norm{J''(\bar{u(\cdot)})^{-1}}_{\mathbb{L}(\LI)}
			\norm{ J''(u(\cdot)) - J''(\bar{u(\cdot)}) }_{\mathbb{L}(\LI)}
			\le \frac{\gamma - \eta}{\gamma}
			< 1.
		\end{aligned}\right.
	\end{equation}
	We also have
	\begin{equation}\label{SODPI.9}
		J''(u(\cdot))
		= J''(\bar{u(\cdot)})
		[
		I
		+ J''(\bar{u(\cdot)})^{-1}
		( J''(u(\cdot)) - J''(\bar{u(\cdot)}) )
		].
	\end{equation}
	Combining \eqref{SODPI.7}--\eqref{SODPI.9} and applying the Neumann series theorem (see \cite[Sec.~2.3.2]{AtkinsonHan2005}), we find that $J''(u(\cdot)): \LI \to \LI$ is invertible and satisfies
	\begin{equation*}
		\norm{J''(u(\cdot))^{-1}}_{\mathbb{L}(\LI)}
		\le \frac{
			\norm{J''(\bar{u(\cdot)})^{-1}}_{\mathbb{L}(\LI)}
		}{
			1- \norm{J''(\bar{u(\cdot)})^{-1}}_{\mathbb{L}(\LI)}
			\norm{J''(u(\cdot)) - J''(\bar{u(\cdot)})}_{\mathbb{L}(\LI)}
		}
		\le \frac{1}{\eta}.
	\end{equation*} 
\end{proof}

\begin{theorem}\label{TQCI}
	\Assume hold. Let $\bar{u(\cdot)} \in \LI$ be an optimal control for DOCP, and suppose that the second-order coercivity condition~\eqref{SSC} holds at $\bar{u(\cdot)}$. If the initial control $u^0(\cdot)$ is sufficiently close to $\bar{u(\cdot)}$ in $\LI$, then the following statements hold:
	\begin{enumerate}
		\item The sequence $\{u^k(\cdot)\}$ generated by \cref{Algorithm-DOCP-Newton} converges quadratically to $\bar{u(\cdot)}$ in $\LI$.
		
		\item A radius $r_3$ that guarantees convergence is given by
		\begin{equation}\label{QCRI}
			r_3
			= \min \left\{
			\frac{\alpha}{C_3(1+\norm{\bar{u}}_1)},
			\frac{
				\sqrt{
					C_3^2 (1+\norm{\bar{u}}_1)^2
					+ C_3 T \gamma
				}
				- C_3 (1+\norm{\bar{u}}_1)
			}{C_3T}
			\right\}
		\end{equation}
		where $C_3$ is given by \eqref{C3} and $\gamma$ by \eqref{gamma-r2}.
		
		\item The gradients ${J'(u^k(\cdot))}$ converge quadratically to zero in $\LI$.
	\end{enumerate}
\end{theorem}

\begin{proof}
	(i) By \cref{TSOD}(ii), the Newton subproblem~\eqref{DOCPk} can be written as the following unconstrained problem:
	\begin{equation}\label{QCI.1}
		\underset{v(\cdot) \in \LT}{\min}
		\inner{J'(u^k(\cdot))}{v(\cdot)}_{\LT}
		+ \frac{1}{2} \inner{J''(u^k(\cdot))v(\cdot)}{v(\cdot)}_{\LT}.
	\end{equation}
	By \cref{TSODC}(i), if $\norm{u^k-\bar{u}}_\infty < \frac{\alpha}{C_3 ( 1+\norm{\bar{u}}_1 )} $, then there exists $\eps > 0$ such that
	\begin{equation*}
		\begin{aligned}
			& \inner{J''(u^k(\cdot))v(\cdot)}{v(\cdot)}_{\LT} \\
			\ge & \inner{J''(\bar{u(\cdot)})v(\cdot)}{v(\cdot)}_{\LT}
			- \abs*{
				\inner{
					( J''(\bar{u(\cdot)})-J''(u^k(\cdot)) )
					v(\cdot)
				}{v(\cdot)}_{\LT}
			} \\
			\ge & [
			\alpha
			- C_3 ( 1+\norm{\bar{u}}_1 )
			\norm{u^k-\bar{u}}_\infty
			] \norm{v}_2^2 \\
			\ge & \eps \norm{v}_2^2.
		\end{aligned}
	\end{equation*}
	In what follows, we assume that $\norm{u^k-\bar{u}}_\infty < \frac{\alpha}{C_3 ( 1+\norm{\bar{u}}_1 )} $. The proof of \cref{TSODP}(i) gives $J''(u^k(\cdot))^{-1} \in \mathbb{L}(\LT)$. The preceding coercivity estimate with constant $\eps$ and \cref{P-convex-positive} yield $H_{uu}^k[t]\ge\eps I$ a.e. Together with \cref{LFS}, these facts imply that problem~\eqref{QCI.1} has the unique optimal solution
	\begin{equation}\label{QCI.2}
		v^k(\cdot)
		= - J''(u^k(\cdot))^{-1} J'(u^k(\cdot)) \in \LI.
	\end{equation}
	Consequently,
	\begin{equation}\label{QCI.3}
		\begin{aligned}
		& u^{k+1}(\cdot) - \bar{u(\cdot)} \\
		& = J''(u^k(\cdot))^{-1}\left[
		J''(u^k(\cdot))
		(u^k(\cdot)-\bar{u(\cdot)})
		- ( J'(u^k(\cdot))-J'(\bar{u(\cdot)}) )
		\right].
		\end{aligned}
	\end{equation}
	Theorem~\ref{TSODC} gives the following estimate. The integral identity along the line segment first holds in $\LT$; the operator-norm continuity in \cref{TSODC}(i) with $\beta=\infty$ then shows that the same identity holds in $\LI$:
	\begin{equation}\label{QCI.4}
		\begin{aligned}
			& \norm{
				J''(u^k(\cdot))(u^k(\cdot)-\bar{u(\cdot)})
				- (J'(u^k(\cdot))-J'(\bar{u(\cdot)}))
			}_\infty \\
			= & \norm*{
				\int_{0}^{1}\left[
				J''(u^k(\cdot))
				- J''((u^k+\theta(\bar{u}-u^k))(\cdot))
				\right](u^k(\cdot)-\bar{u(\cdot)}) \dd \theta
			}_\infty \\
			\le & \int_{0}^{1}\norm*{
				\left[
				J''(u^k(\cdot))
				- J''((u^k+\theta(\bar{u}-u^k))(\cdot))
				\right](u^k(\cdot)-\bar{u(\cdot)})
			}_\infty \dd \theta \\
			\le & \int_{0}^{1} \theta C_3 (1+\norm{u^k}_1)
			\norm{u^k-\bar{u}}_\infty^2 \dd \theta \\
			\le & \frac{1}{2}C_3(1+\norm{\bar{u}}_1+\norm{\bar{u}-u^k}_1)
			\norm{u^k-\bar{u}}_\infty^2.
		\end{aligned}
	\end{equation}
	For any $\eta \in (0,\gamma)$, if $u^k(\cdot)\in B_{\LI}(\bar{u},r_2)$, then combining \eqref{QCI.1}--\eqref{QCI.4} gives
	\begin{align}
		\notag
		\norm{u^{k+1}(\cdot) - \bar{u(\cdot)}}_\infty
		\le & \norm{
			J''(u^k(\cdot))^{-1}
		}_{\mathbb{L}(\LI)}
		\norm{
			J''(u^k(\cdot))
			(u^k(\cdot)-\bar{u(\cdot)})
			- ( J'(u^k(\cdot))-J'(\bar{u(\cdot)}) )
		}_\infty \\
		\label{QCI.5}
		\le & \frac{C_3}{2\eta}(1+\norm{\bar{u}}_1+\norm{\bar{u}-u^k}_1)
		\norm{u^k-\bar{u}}_\infty^2 \\
		\notag
		\le & \frac{C_3}{2\eta}(1+\norm{\bar{u}}_1+Tr_2)
		\norm{u^k-\bar{u}}_\infty^2.
	\end{align}
	Therefore, if there exists $\eta\in(0,\gamma)$ such that the initial control $u^0(\cdot)$ satisfies
	
	\begin{equation}\label{QCI.6}
		\norm{u^0-\bar{u}}_\infty
		\le \min \left\{
		\frac{\alpha}{C_3 ( 1+\norm{\bar{u}}_1 )},
		r_2,
		\frac{\eta}{C_3(1+\norm{\bar{u}}_1+Tr_2)}\right\}
		\triangleq \delta(\eta),
	\end{equation}
	
	then all iterates remain in the neighborhood $B_{ \LI}(\bar{u},\delta(\eta))$, and the sequence $\{u^k(\cdot)\}$ generated by Algorithm~\ref{Algorithm-DOCP-Newton} converges quadratically to $\bar{u(\cdot)}$ in $\LI$.
	
	(ii) The maximum value of $\delta(\eta)$ gives the convergence radius guaranteed by the preceding sufficient condition. Substituting the definition of $r_2$ in \eqref{gamma-r2} into \eqref{QCI.6} yields
	\begin{equation}\label{QCI.7}
		\delta(\eta)
		= \min \left\{
		\frac{\alpha}{C_3 ( 1+\norm{\bar{u}}_1 )},
		\frac{\gamma-\eta}{C_3(1+\norm{\bar{u}}_1)},
		\frac{C_3(1+\norm{\bar{u}}_1)\eta}
		{ C_3^2(1+\norm{\bar{u}}_1)^2 + C_3T(\gamma-\eta) }
		\right\}, ~ \eta\in(0,\gamma).
	\end{equation}
	Of the three terms on the right-hand side, the first is independent of $\eta$. The second is strictly decreasing in $\eta$ and tends to $0$ as $\eta\to \gamma^-$, whereas the third is strictly increasing in $\eta$ and tends to $0$ as $\eta\to 0^+$. These two functions therefore have a unique intersection in $(0,\gamma)$, at which $\delta(\cdot)$ attains its maximum. Set
	\begin{equation*}
		\frac{\gamma-\eta}{C_3(1+\norm{\bar{u}}_1)}
		= \frac{C_3(1+\norm{\bar{u}}_1)\eta}
		{(1+\norm{\bar{u}}_1)^2C_3^2+C_3T(\gamma-\eta)},
	\end{equation*}
	which yields
	\begin{equation*}
		\eta_*
		= \frac{
			\gamma\sqrt{C_3^2(1+\norm{\bar{u}}_1)^2+C_3T\gamma}
		}{
			\sqrt{C_3^2(1+\norm{\bar{u}}_1)^2+C_3T\gamma}
			+C_3(1+\norm{\bar{u}}_1)
		}.
	\end{equation*}
	Since $\gamma\le\alpha$, the corresponding maximum is
	\begin{equation*}
		\max_{\eta\in(0,\gamma)}\delta(\eta)
		= \frac{
			\sqrt{C_3^2(1+\norm{\bar{u}}_1)^2+C_3T\gamma}
			-C_3(1+\norm{\bar{u}}_1)
		}{C_3T},
	\end{equation*}
	which gives \eqref{QCRI}.
	
	(iii) For any $\eta \in (0,\gamma)$, if $u^k(\cdot) \in B_{\LI}(\bar{u},r_2)$, then the argument used to derive \eqref{QCI.4}, together with \eqref{QCI.2}, gives
	\begin{align*}
		& \norm{J'(u^{k+1}(\cdot))}_\infty \\
		=	& \norm{
			J'(u^{k+1}(\cdot))-J'(u^{k}(\cdot))
			- J''(u^{k}(\cdot))v^k(\cdot)
		}_\infty \\
		=	& \norm*{
			\int_{0}^{1}J''((u^k+\theta v^k)(\cdot))v^k(\cdot)\dd \theta
			- J''(u^k(\cdot))v^k(\cdot)
		}_\infty \\
		\le	& \int_{0}^{1}\norm{
			(
			J''((u^k+\theta v^k)(\cdot))
			- J''(u^k(\cdot))
			) v^k(\cdot)
		}_\infty \dd \theta \\
		\le & \int_{0}^{1} \theta C_3
		(1+\norm{u^k}_1)
		\norm{v^k}_\infty^2 \dd \theta \\
		=	& \frac{1}{2}C_3(1+\norm{u^k}_1)
		\norm{
			J''(u^k(\cdot))^{-1} J'(u^k(\cdot))
		}_\infty^2 \\
		\le & \frac{1}{2 \eta^2} C_3 (1+\norm{u^k}_1)
		\norm{J'(u^k(\cdot))}_\infty^2.
	\end{align*}
	By (i), $1+\norm{u^k}_1\le1+\norm{\bar u}_1+Tr_2$. Equation~\eqref{FOD} and the same pointwise difference estimate as in \eqref{EH.2} give
	\begin{equation*}
		\norm{J'(u^k(\cdot))-J'(\bar{u(\cdot)})}_\infty
		\le C_1(1+\norm{u^k}_1)(T+1)\norm{u^k-\bar u}_\infty\to0.
	\end{equation*}
	Hence the coefficient in the preceding quadratic estimate can be bounded by a constant independent of $k$, and $\{J'(u^k(\cdot))\}$ converges quadratically to zero in $\LI$.
\end{proof}

\begin{theorem}\label{TQCT}
	\Assume hold, and $f_{uu},b_{uu}$ are independent of $u$. Let $\bar{u(\cdot)}\in \LT$ be an optimal control for problem~\eqref{DOCP}, and let $\bar{x}(\cdot)$ and $\bar{p}(\cdot)$ be the corresponding state and adjoint variables, respectively. If the second-order coercivity condition~\eqref{SSC} holds at $\bar{u(\cdot)}$ and the initial control $u^0(\cdot)$ is sufficiently close to $\bar{u(\cdot)}$ in $ \LT$, then the following statements hold:
	\begin{enumerate}
		\item The sequence $\{u^k(\cdot)\}$ generated by Algorithm~\ref{Algorithm-DOCP-Newton} converges quadratically to $\bar{u(\cdot)}$ in $ \LT$.
		
		\item A radius $r_4$ that guarantees convergence is given by
		\begin{equation}\label{QCRT}
			r_4
			= \frac{
				\sqrt{
					C_4^2 (1+\norm{\bar{u}}_1)^2
					+ C_4 \sqrt{T} \alpha
				}
				- C_4 (1+\norm{\bar{u}}_1)
			}{C_4\sqrt{T}},
		\end{equation}
		where $C_4$ is the value in \eqref{C4} for $\beta=2$, and $\alpha$ is given by \eqref{SSC}.
		
		\item The gradients $\{J'(u^k(\cdot))\}$ converge quadratically to zero in $\LT$.
	\end{enumerate}
\end{theorem}

\begin{proof}
	(i) For any $\eta \in (0,\alpha)$, \cref{TSODP} shows that, if $u^k(\cdot) \in B_{\LT}(\bar{u},r_1)$, then the Newton subproblem~\eqref{DOCPk} has the unique solution
	\begin{equation}\label{QCT.1}
		v^k(\cdot)
		= - J''(u^k(\cdot))^{-1} J'(u^k(\cdot))
		\in \LT.
	\end{equation}
	By \cref{TSODC},
	\begin{align}
		\notag
		& \norm{
			J''(u^k(\cdot)) (u^k(\cdot)-\bar{u(\cdot)})
			- ( J'(u^k(\cdot))-J'(\bar{u(\cdot)}) )
		}_2  \\
		\notag
		=	& \norm*{
		J''(u^k(\cdot))
		(u^k(\cdot)-\bar{u(\cdot)})
		- \int_{0}^{1}J''((u^k+\theta(\bar{u}-u^k))(\cdot))
		(u^k(\cdot)-\bar{u(\cdot)}) \dd \theta
		}_2 \\
		\notag
		=	& \left\{\int_{0}^{T}\dd t\abs*{
		\int_{0}^{1}\Bigl[
		\Bigl(
		J''(u^k(\cdot))
		- J''((u^k+\theta(\bar{u}-u^k))(\cdot))
		\Bigr)
		(u^k(\cdot)-\bar{u(\cdot)})
		\Bigr](t) \dd \theta
		}^2\right\}^\frac{1}{2} \\
		\notag
		\overset{(a)}{\le}
		& \int_{0}^{1}\norm*{
			\left[
			J''(u^k(\cdot))
			- J''((u^k+\theta(\bar{u}-u^k))(\cdot))
			\right]
			(u^k(\cdot)-\bar{u(\cdot)})
		}_2 \dd \theta \\
		\notag
		\le & \int_{0}^{1} \theta C_4 (1+\norm{u^k}_1)
		\norm{u^k-\bar{u}}_2^2 \dd \theta \\
		\notag
		= & \frac{1}{2}C_4(1+\norm{u^k}_1) \norm{u^k-\bar{u}}_2^2 \\
		\label{QCT.2}
		\le & \frac{1}{2}C_4(1+\norm{\bar{u}}_1+\norm{\bar{u}-u^k}_1)
		\norm{u^k-\bar{u}}_2^2,
	\end{align}
	where $(a)$ follows from Minkowski's integral inequality. If $u^k(\cdot) \in B_{ \LT}(\bar{u},r_1)$, then combining \eqref{QCT.1}--\eqref{QCT.2} gives
	\begin{align*}
		\norm{u^{k+1}-\bar{u}}_2
		\le & \norm{
			J''(u^k(\cdot))^{-1}
		}_{\mathbb{L}(\LT)}
		\norm{
			J''(u^k(\cdot))
			(u^k(\cdot)-\bar{u(\cdot)})
			- ( J'(u^k(\cdot)) - J'(\bar{u(\cdot)}) )
		}_2   \\
		\le & \frac{C_4}{2\eta}
		( 1 + \norm{\bar{u}}_1 + \norm{u^k-\bar{u}}_1 )
		\norm{u^k-\bar{u}}_2^2 \\
		\le & \frac{C_4}{2\eta}
		(1+\norm{\bar{u}}_1+\sqrt{T}r_1)
		\norm{u^k-\bar{u}}_2^2.
	\end{align*}
	Therefore, if there exists $\eta \in (0,\alpha)$ such that the initial control $u^0(\cdot)$ satisfies
	
	\begin{equation}\label{QCT.3}
		\norm{u^0-\bar{u}}_2
		\le \min\left\{r_1,\frac{\eta}{C_4(1+\norm{\bar{u}}_1+\sqrt{T}r_1)}\right\}
		\triangleq \delta(\eta),
	\end{equation}
	
	then all iterates remain in the neighborhood $B_{ \LT}(\bar{u},\delta(\eta))$, and the sequence $\{u^k(\cdot)\}$ generated by Algorithm~\ref{Algorithm-DOCP-Newton} converges quadratically to $\bar{u(\cdot)}$ in $\LT$.
	
	(ii) As in \cref{TQCI}(ii), at the unique intersection of the two terms in \eqref{QCT.3} within $(0,\alpha)$, we have
	\begin{equation*}
		\alpha-\eta=C_4(1+\norm{\bar u}_1)r_1,\qquad
		\eta=C_4r_1(1+\norm{\bar u}_1+\sqrt{T}r_1).
	\end{equation*}
	Thus $C_4\sqrt{T}r_1^2+2C_4(1+\norm{\bar u}_1)r_1=\alpha$, and its positive root gives \eqref{QCRT}.
	
	\begingroup
	\interdisplaylinepenalty=10000
	(iii) By \eqref{QCT.1}, arguing as in the proof of \cref{TQCI}(iii), we obtain
	\begin{align*}
		&\norm{J'(u^{k+1}(\cdot))}_2\\
		=&\norm*{\int_0^1
		[J''((u^k+\theta v^k)(\cdot))-J''(u^k(\cdot))]
		v^k(\cdot)\dd\theta}_2\\
		\le&\frac{C_4}{2}(1+\norm{u^k}_1)\norm{v^k}_2^2\\
		\le&\frac{C_4}{2\eta^2}(1+\norm{\bar u}_1+\sqrt{T}r_1)
		\norm{J'(u^k(\cdot))}_2^2.
	\end{align*}
	\endgroup
	By (i) and the continuity of $J'$, we have $J'(u^k(\cdot))\to0$. Hence the gradients converge quadratically to $0$ in $\LT$.
\end{proof}

\subsection{Numerical examples}

We first consider a problem to which Theorem~\ref{TQCI} applies but which does not satisfy the additional condition in Theorem~\ref{TQCT}.

\begin{example}\label{ex1}
	Let $T > 0$. Consider the multidimensional extension of the optimal control problem in Example~\ref{E-continuous}:
	\begin{equation}
		\begin{aligned}
			& \underset{u(\cdot)\in  \LT}{\mathrm{min}}~
			J(u(\cdot))
			= \frac{1}{2}\int_{0}^{T}\left\{
			\abs*{x(t)}^2 + \abs*{\sin(u(t))}^2
			\right\} \dd t \\
			& \hspace{2em}\mathrm{s.t.} \hspace{2.5em}
			\left\{
			\begin{aligned}
				& \dot{x}(t)
				= u(t), \quad t\in[0,T], \\
				& x(0) = 0,
			\end{aligned}
			\right.
		\end{aligned}
	\end{equation}
	Here $\sin(u(t)) = (\sin(u_1(t)),\cdots,\sin(u_m(t)))$. This problem satisfies Assumptions~\ref{A1}--\ref{A2}, but $f_{uu}$ depends on $u$, so the additional condition in Theorem~\ref{TQCT} fails. Moreover, $J'':\LT\to\mathbb{L}(\LT)$ is discontinuous at the zero control, and $\bar{u(\cdot)} \equiv 0$ is an optimal control (see Example~\ref{E-continuous}). Let $\bar{x}(\cdot)$ and $\bar{p}(\cdot)$ be the corresponding state and adjoint variables, respectively. Then
	\begin{align*}
		H_{(x,u)^2}(t,\bar{x}(t),\bar{u}(t),\bar{p}(t))
		= I_{2m},
	\end{align*}
	where $I_{2m}$ is the identity matrix of order $2m$. Thus the second-order coercivity condition~\eqref{SSC} holds, and Newton's method (Algorithm~\ref{Algorithm-DOCP-Newton}) can be applied. In this example, we measure errors using $\norm{\cdot}_\infty$. We choose the initial control $u_i^0(t)\equiv\xi/(2T)$ ($i=1,\ldots,m$), where $\xi=\mathrm{rand}(1)\in(0,1)$ is a single random number shared by all components. For $T=1, m=10$, we obtain the following results:
	\begin{figure}[htbp]
		\centering
		\begin{minipage}[t]{0.48\textwidth}
			\centering
			\includegraphics[width=\textwidth]{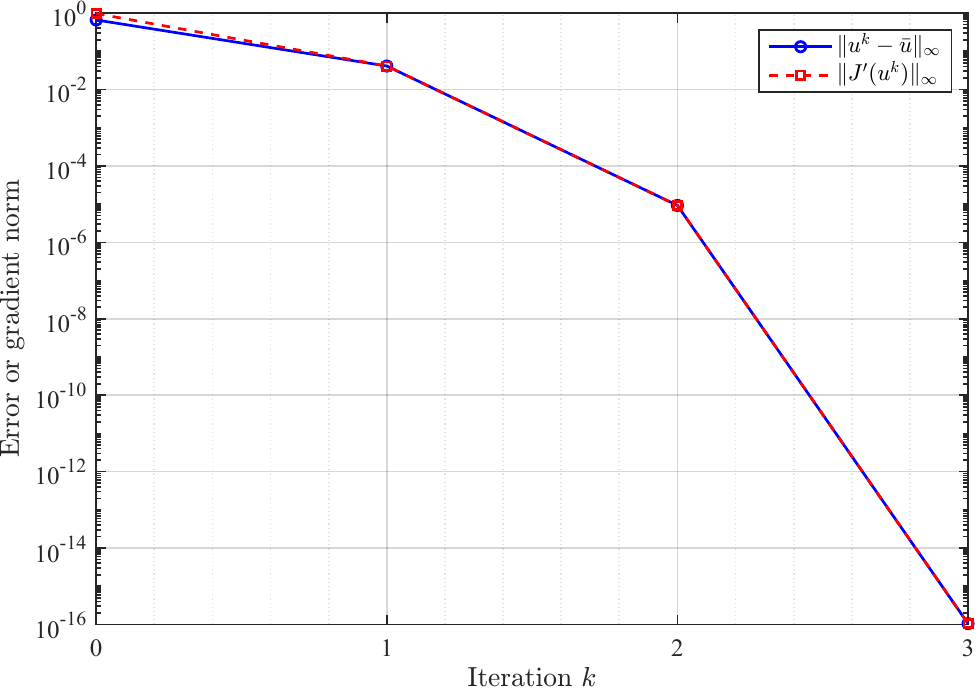}
			\caption{Example~\ref{ex1}. Control error and gradient norm ($T=1$, $m=10$).}
			\label{f1}
		\end{minipage}
		\hfill
		\begin{minipage}[t]{0.48\textwidth}
			\centering
			\includegraphics[width=\textwidth]{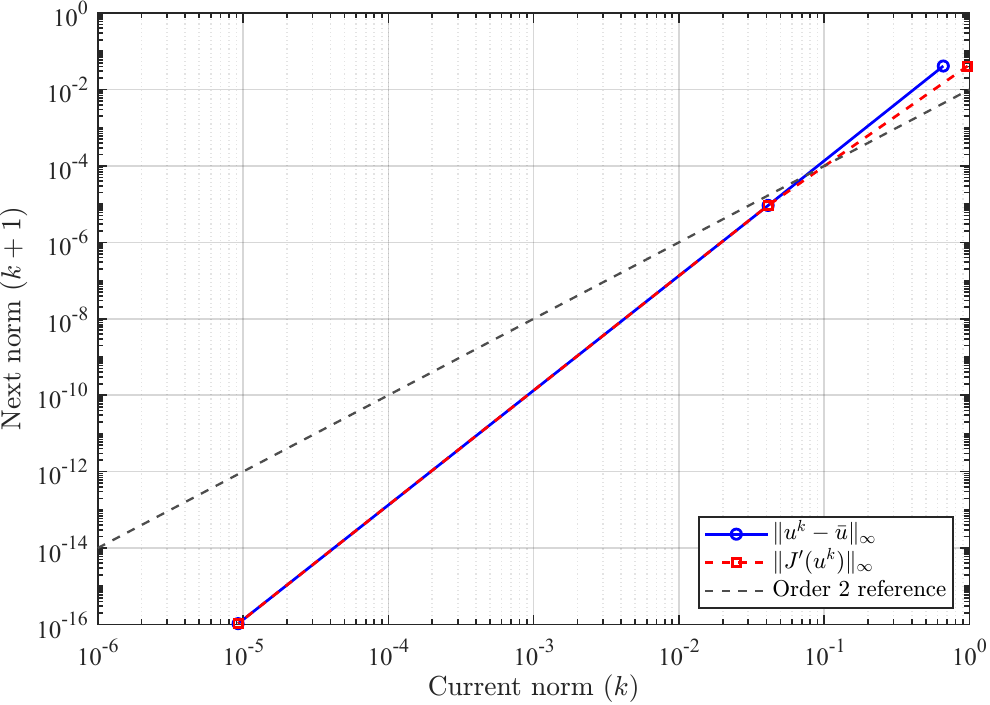}
			\caption{Example~\ref{ex1}. Successive errors and gradient norms ($T=1$, $m=10$).}
			\label{f2}
		\end{minipage}
	\end{figure}
	
	Figure~\ref{f1} presents semilogarithmic plots of the control error $\norm{u^k-\bar{u}}_\infty$ and the gradient norm $\norm{J'(u^k(\cdot))}_\infty$ against the iteration number. Both decrease rapidly and fall below $10^{-12}$ after $3$ Newton iterations. Figure~\ref{f2} shows log--log plots of successive control errors and gradient norms; the gray dashed line is a reference line of slope $2$. The observed decrease is consistent with quadratic or faster convergence, providing numerical support for the local quadratic convergence result. For $T=3, m=10$, the same accuracy is also reached after $3$ iterations, with similar convergence behavior shown in Figures~\ref{f3}--\ref{f4}:
	\begin{figure}[htbp]
		\begin{minipage}{0.48\textwidth}
			\centering
			\includegraphics[width=\textwidth]{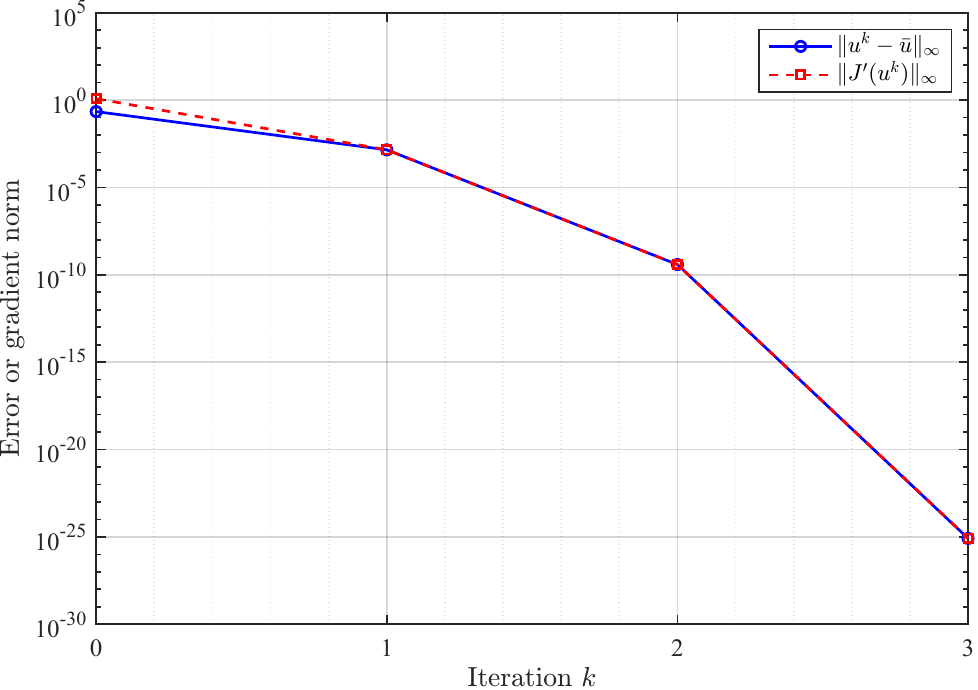}
			\caption{Example~\ref{ex1}. Control error and gradient norm ($T=3$, $m=10$).}
			\label{f3}
		\end{minipage}
		\hfill
		\begin{minipage}{0.48\textwidth}
			\centering
			\includegraphics[width=\textwidth]{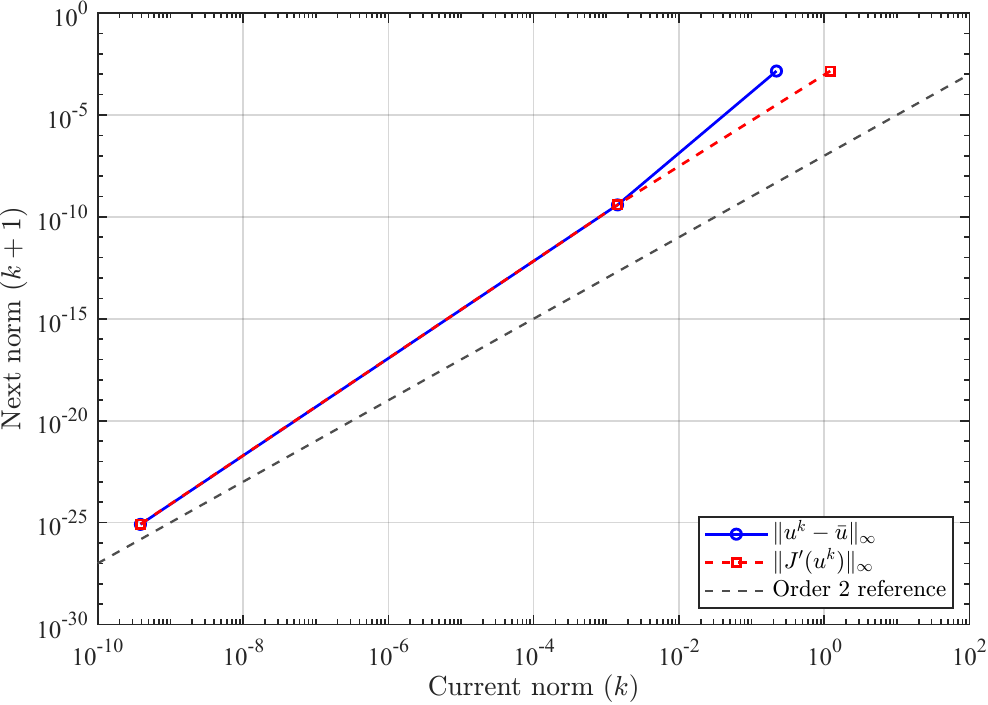}
			\caption{Example~\ref{ex1}. Successive errors and gradient norms ($T=3$, $m=10$).}
			\label{f4}
		\end{minipage}
	\end{figure}
\end{example}

We next consider an example to which Theorem~\ref{TQCT} applies.
\begin{example}\label{ex2}
	Let $T > 0$. Consider the following optimal control problem:
	\begin{equation}
		\begin{aligned}
			& \underset{u(\cdot)\in  \LT}{\mathrm{min}}~
			J(u(\cdot))
			= \frac{1}{2}\int_{0}^{T}\left\{
			\abs*{x(t)}^2 + \abs*{u(t)}^2
			\right\} \dd t \\
			& \hspace{2em}\mathrm{s.t.} \hspace{2.5em}
			\left\{
			\begin{aligned}
				& \dot{x}_i(t)
				= \sin(x_i(t)) + u_i(t), ~ i = 1,2,\cdots,m, \quad t\in[0,T], \\
				& x(0) = 0,
			\end{aligned}
			\right.
		\end{aligned}
	\end{equation}
	Here $x_i$ and $u_i$ denote the $i$th components of $x$ and $u$, respectively ($x,u\in\RR^m$). This problem satisfies Assumptions~\ref{A1}--\ref{A2}. Moreover, $b_{uu}=0$ and $f_{uu}=I_m$ are independent of $u$, so the additional structural condition in Theorem~\ref{TQCT} is satisfied. The control $\bar{u(\cdot)} \equiv 0$ is optimal. Let $\bar{x}(\cdot)$ and $\bar{p}(\cdot)$ be the corresponding state and adjoint variables, respectively. Then
	\begin{align*}
		H_{(x,u)^2}(t,\bar{x}(t),\bar{u}(t),\bar{p}(t))
		= I_{2m},
	\end{align*}
	where $I_{2m}$ is the identity matrix of order $2m$. Thus the second-order coercivity condition~\eqref{SSC} holds, and Newton's method can be applied. In this example, we measure errors using $\norm{\cdot}_2$. We choose the initial control $u_i^0(t)\equiv\xi/T^2$ ($i=1,\ldots,m$), with the same $\xi$ as in Example~\ref{ex1}. For $T=1, m=10$ and $T=3, m=10$, we obtain the following results:
	\begin{figure}[htbp]
		\centering
		\begin{minipage}[t]{0.48\textwidth}
			\centering
			\includegraphics[width=\textwidth]{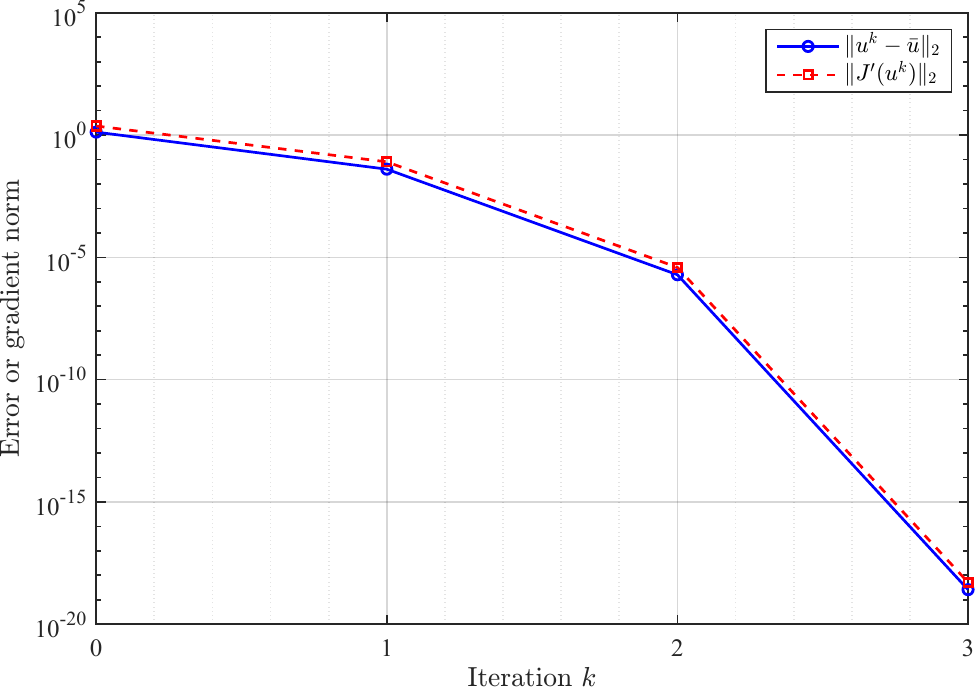}
			\caption{Example~\ref{ex2}. Control error and gradient norm ($T=1$, $m=10$).}
			\label{f5}
		\end{minipage}
		\hfill
		\begin{minipage}[t]{0.48\textwidth}
			\centering
			\includegraphics[width=\textwidth]{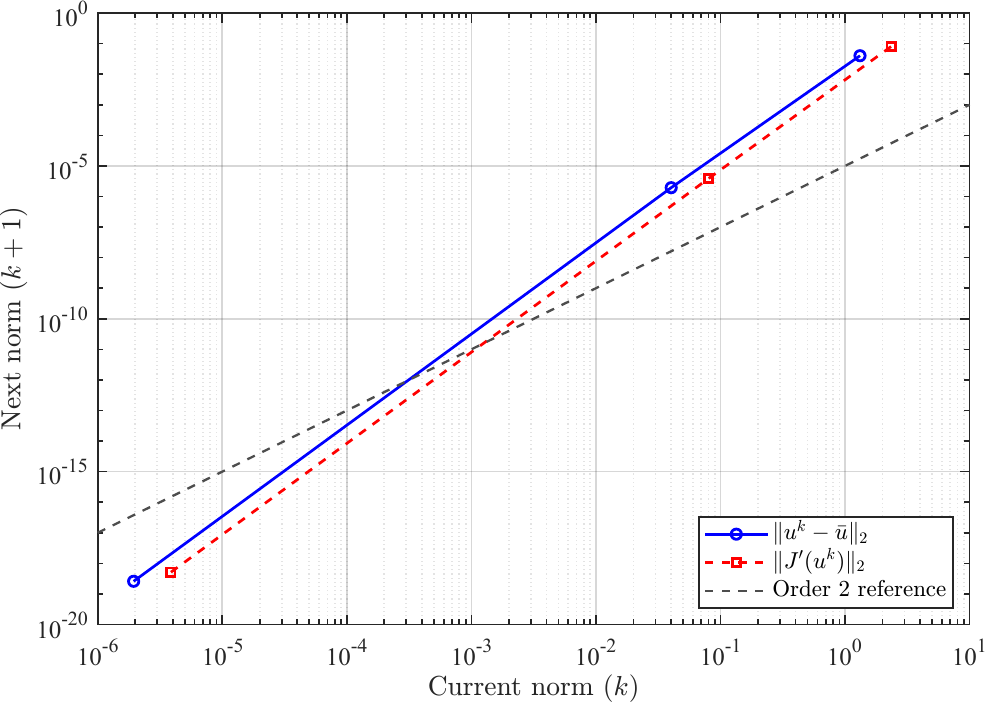}
			\caption{Example~\ref{ex2}. Successive errors and gradient norms ($T=1$, $m=10$).}
			\label{f6}
		\end{minipage}
	\end{figure}
	
	\begin{figure}[htbp]
		\centering
		\begin{minipage}[t]{0.48\textwidth}
			\centering
			\includegraphics[width=\textwidth]{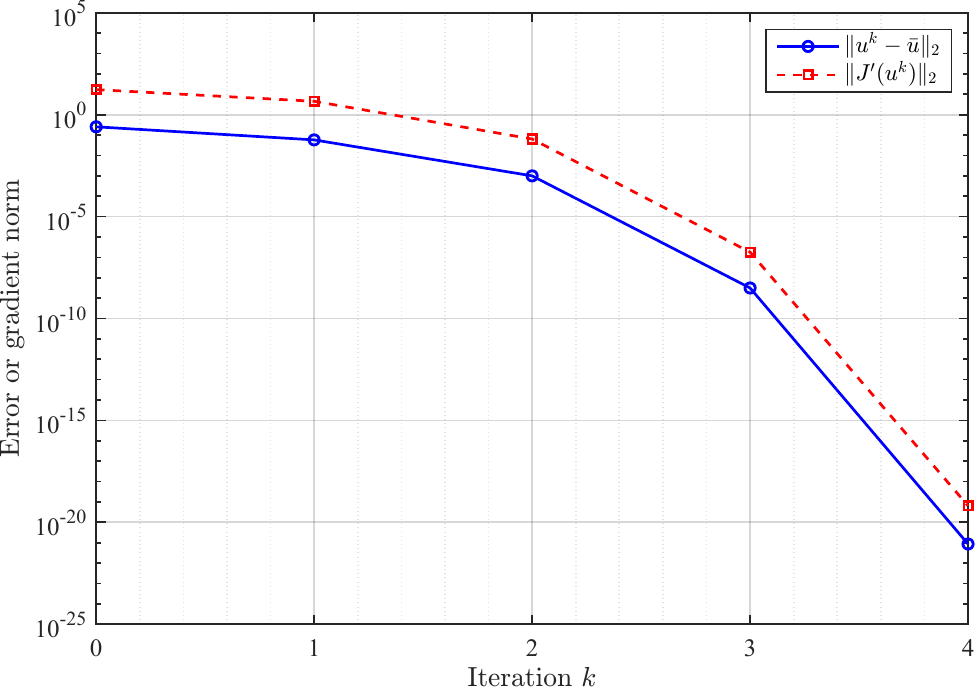}
			\caption{Example~\ref{ex2}. Control error and gradient norm ($T=3$, $m=10$).}
			\label{f7}
		\end{minipage}
		\hfill
		\begin{minipage}[t]{0.48\textwidth}
			\centering
			\includegraphics[width=\textwidth]{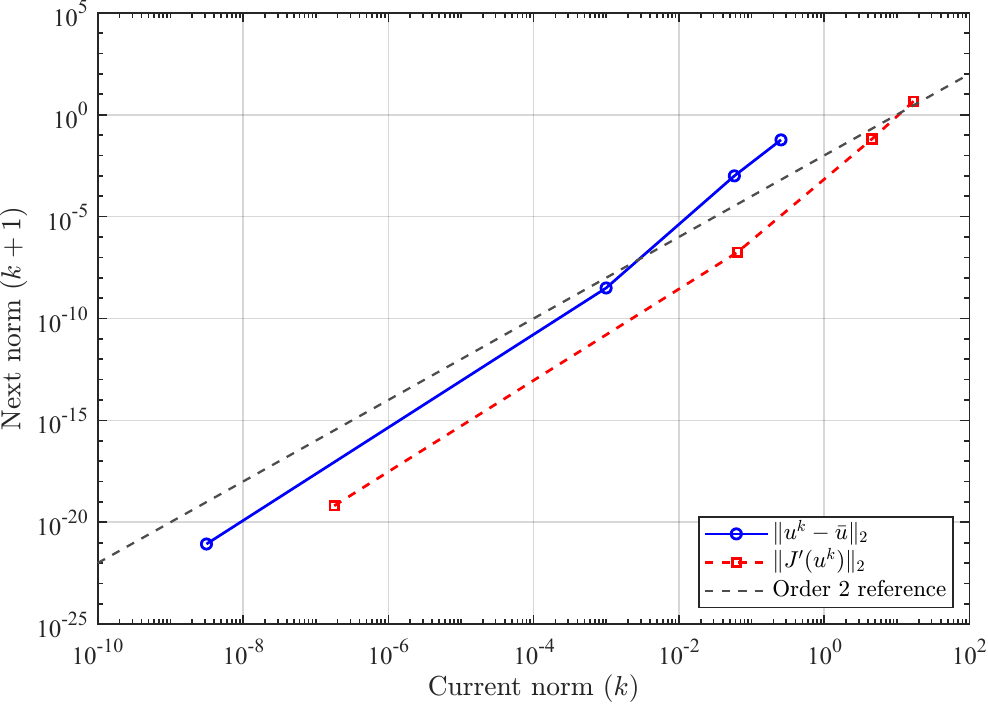}
			\caption{Example~\ref{ex2}. Successive errors and gradient norms ($T=3$, $m=10$).}
			\label{f8}
		\end{minipage}
	\end{figure}
	
	Figures~\ref{f5}--\ref{f8} show semilogarithmic plots of the control error $\norm{u^k-\bar{u}}_2$ and the gradient norm $\norm{J'(u^k(\cdot))}_2$, together with log--log plots of their values at successive iterations. For $T=1$ and $T=3$, both quantities fall below $10^{-12}$ after $3$ and $4$ Newton iterations, respectively. The log--log plots again show a trend consistent with quadratic or faster convergence, providing numerical support for the local quadratic convergence result in Theorem~\ref{TQCT}.
\end{example}

\newpage
\bibliographystyle{gbt7714-numerical}  
\bibliography{Reference}

@book{Anderson2007,
	title={Optimal control: linear quadratic methods},
	author = {Anderson, Brian D. O. and Moore, John B.},
	year={2007},
	publisher = {Dover Publications},
	address = {Mineola, NY},
	edition = {Augmented}
}

@article{Alt1990,
	author    = {Alt, Walter},
	title     = {The {Lagrange-Newton} method for infinite-dimensional optimization problems},
	journal   = {Numerical Functional Analysis and Optimization},
	year      = {1990},
	volume    = {11},
	number    = {3-4},
	pages = {201--224},
	doi = {10.1080/01630569008816371}
}

@article{Alt1994,
	author    = {Alt, Walter},
	title     = {Local convergence of the {Lagrange-Newton} method with applications to optimal control},
	journal   = {Control and Cybernetics},
	year      = {1994},
	volume    = {23},
	number    = {1-2},
	pages = {87--105}
}

@article{AltMalanowski1993,
	author    = {Alt, Walter and Malanowski, Kazimierz},
	title     = {The {Lagrange-Newton} method for nonlinear optimal control problems},
	journal   = {Computational Optimization and Applications},
	year      = {1993},
	volume    = {2},
	pages = {77--100},
	doi = {10.1007/bf01299143},
	number = {1}
}

@book{AtkinsonHan2005,
	author    = {Atkinson, Kendall and Han, Weimin},
	title = {Theoretical Numerical Analysis: A Functional Analysis Framework},
	year      = {2005},
	publisher = {Springer},
	address = {New York},
	doi = {10.1007/978-0-387-28769-0},
	edition = {2}
}

@article{Blanes2024,
	title = {Splitting methods for differential equations},
	author={Blanes, Sergio and Casas, Fernando and Murua, Ander},
	journal={Acta Numerica},
	volume={33},
	year={2024},
	doi = {10.1017/s0962492923000077},
	pages = {1--161}
}

@article{Boggs1995,
	author    = {Boggs, Paul T. and Tolle, Jon W.},
	title     = {Sequential quadratic programming},
	journal   = {Acta Numerica},
	year      = {1995},
	volume    = {4},
	pages = {1--51},
	doi = {10.1017/s0962492900002518}
}

@article{BonnansSilva2012,
	author    = {Bonnans, J. Fr{\'e}d{\'e}ric and Silva, Francisco J.},
	title     = {First and second order necessary conditions for stochastic optimal control problems},
	journal = {Applied Mathematics \& Optimization},
	year      = {2012},
	volume    = {65},
	pages = {403--439},
	doi = {10.1007/s00245-012-9162-4},
	number = {3}
}

@article{CasasMateos2026,
	title = {Quadratic convergence of an {SQP} method for some optimization problems with applications to control theory},
	author={Casas, Eduardo and Mateos, Mariano},
	journal={SIAM Journal on Control and Optimization},
	volume={64},
	number={3},
	pages={1127--1149},
	year={2026},
	publisher={SIAM},
	doi = {10.1137/25m176533x}
}

@inproceedings{Choi1989,
	title = {Efficient matrix-valued algorithms for solving stiff {Riccati} differential equations},
	author = {Choi, Chiu H. and Laub, Alan J.},
	booktitle = {Proceedings of the 28th IEEE Conference on Decision and Control},
	pages={885--887},
	year={1989},
	organization={IEEE},
	doi = {10.1109/cdc.1989.70248}
}

@article{Dieci1994,
	title = {Positive definiteness in the numerical solution of {Riccati} differential equations},
	author={Dieci, Luca and Eirola, Timo},
	journal={Numerische Mathematik},
	volume={67},
	number={3},
	pages={303--313},
	year={1994},
	publisher={Springer},
	doi = {10.1007/s002110050030}
}

@article{DontchevHagerPoore1995,
	title={Optimality, stability, and convergence in nonlinear control},
	author = {Dontchev, A. L. and Hager, W. W. and Poore, A. B. and Yang, Bing},
	journal = {Applied Mathematics \& Optimization},
	volume={31},
	number={3},
	pages={297--326},
	year={1995},
	publisher={Springer},
	doi = {10.1007/bf01215994}
}

@article{PalonmaresMagasatian1976,
	author    = {Garcia-Palomares, Ubaldo M. and Mangasarian, Olvi L.},
	title     = {Superlinearly convergent quasi-{Newton} algorithms for nonlinearly constrained optimization problems},
	journal   = {Mathematical Programming},
	year      = {1976},
	volume    = {11},
	number    = {1},
	pages = {1--13},
	doi = {10.1007/bf01580366}
}

@article{GobetGrangereau2022,
	author = {Gobet, Emmanuel and Grangereau, Maxime},
	title     = {{Newton} method for stochastic control problems},
	journal   = {SIAM Journal on Control and Optimization},
	year      = {2022},
	volume    = {60},
	number    = {5},
	pages = {2996--3025},
	doi = {10.1137/21m1408567}
}

@article{Han1976,
	author    = {Han, Shih-Ping},
	title     = {Superlinearly convergent variable metric algorithms for general nonlinear programming problems},
	journal   = {Mathematical Programming},
	year      = {1976},
	volume    = {11},
	number    = {1},
	pages = {263--282},
	doi = {10.1007/bf01580395}
}

@article{HehlNeitzel2024,
	title = {Local quadratic convergence of the {SQP} method for an optimal control problem governed by a regularized fracture propagation model},
	author={Hehl, Andreas and Neitzel, Ira},
	journal={ESAIM: Control, Optimisation and Calculus of Variations},
	volume={30},
	pages={68},
	year={2024},
	publisher={EDP Sciences},
	doi = {10.1051/cocv/2024052}
}

@article{Heinkenschloss1996,
	author    = {Heinkenschloss, Matthias},
	title     = {Projected sequential quadratic programming methods},
	journal   = {SIAM Journal on Optimization},
	year      = {1996},
	volume    = {6},
	number    = {2},
	pages = {373--417},
	doi = {10.1137/0806022}
}

@article{HeinkenschlossTroltzsch1998,
	title = {Analysis of the {Lagrange-SQP-Newton} method for the control of a phase field equation},
	author={Heinkenschloss, Matthias and Tr{\"o}ltzsch, Fredi},
	year = {1999},
	journal = {Control and Cybernetics},
	volume = {28},
	number = {2},
	pages = {177--211}
}

@article{HinzeKunisch2001,
	author    = {Hinze, Michael and Kunisch, Karl},
	title     = {Second order methods for optimal control of time-dependent fluid flow},
	journal   = {SIAM Journal on Control and Optimization},
	year      = {2001},
	volume    = {40},
	number    = {3},
	pages = {925--946},
	doi = {10.1137/s0363012999361810}
}

@article{HoppeNeitzel2021,
	author = {Hoppe, Fabian and Neitzel, Ira},
	title     = {Convergence of the {SQP} method for quasilinear parabolic optimal control problems},
	journal   = {Optimization and Engineering},
	year      = {2021},
	volume    = {22},
	number    = {4},
	pages = {2039--2085},
	doi = {10.1007/s11081-020-09547-2}
}

@article{ItoKunisch2000,
	author    = {Ito, Kazufumi and Kunisch, Karl},
	title     = {{Newton's} method for a class of weakly singular optimal control problems},
	journal   = {SIAM Journal on Optimization},
	year      = {2000},
	volume    = {10},
	number    = {3},
	pages = {896--916},
	doi = {10.1137/s1052623497320840}
}

@article{Kalman1960,
	title={Contributions to the theory of optimal control},
	author={Kalman, Rudolf Emil},
	journal = {Bolet{\'i}n de la Sociedad Matem{\'a}tica Mexicana (2)},
	volume={5},
	pages={102--119},
	year={1960}
}

@inproceedings{KellyKoppMoyer1963,
	author    = {Kelley, Henry J. and Kopp, Richard E. and Moyer, H. Gardner},
	title     = {A trajectory optimization technique based upon the theory of the second variation},
	booktitle = {Astrodynamics Conference},
	year      = {1963},
	doi = {10.2514/6.1963-415},
	note = {AIAA Paper 63-415},
	publisher = {American Institute of Aeronautics and Astronautics}
}

@article{Kleinman1968,
	title = {On an iterative technique for {Riccati} equation computations},
	author={Kleinman, David},
	journal={IEEE Transactions on Automatic Control},
	volume={13},
	number={1},
	pages={114--115},
	year={1968},
	publisher={IEEE},
	doi = {10.1109/tac.1968.1098829}
}

@incollection{KoppMcGill1964,
	author    = {Kopp, Richard E. and McGill, Robert},
	title     = {Several trajectory optimization techniques: {Part I}: Discussion},
	booktitle = {Computing Methods in Optimization Problems},
	publisher = {Academic Press},
	year      = {1964},
	pages = {65--89},
	doi = {10.1016/b978-1-4831-9812-5.50007-2}
}

@article{Kuvcera1973,
	title = {A review of the matrix {Riccati} equation},
	author={Ku{\v{c}}era, Vladim{\'\i}r},
	journal={Kybernetika},
	volume={9},
	number={1},
	pages={42--61},
	year={1973},
	publisher={Institute of Information Theory and Automation AS CR}
}

@article{Laub1979,
	title = {A {Schur} method for solving algebraic {Riccati} equations},
	author = {Laub, Alan J.},
	journal = {IEEE Transactions on Automatic Control},
	volume={24},
	number={6},
	pages={913--921},
	year={1979},
	publisher={IEEE},
	doi = {10.1109/tac.1979.1102178}
}

@article{Lang2015,
	title = {On the benefits of the {$LDL^{T}$} factorization for large-scale differential matrix equation solvers},
	author={Lang, Norman and Mena, Hermann and Saak, Jens},
	journal={Linear Algebra and its Applications},
	volume={480},
	pages={44--71},
	year={2015},
	publisher={Elsevier},
	doi = {10.1016/j.laa.2015.04.006}
}

@book{Machielsen1988,
	author    = {Machielsen, K. C. P.},
	title     = {Numerical solution of optimal control problems with state constraints by sequential quadratic programming in function space},
	year      = {1988},
	publisher = {Centrum voor Wiskunde en Informatica},
	address = {Amsterdam},
	series = {{CWI Tract}},
	number = {53}
}

@article{Malanowski1993,
	author    = {Malanowski, Kazimierz},
	title     = {Two-norm approach in stability and sensitivity analysis of optimization and optimal control problems},
	journal   = {Advances in Mathematical Sciences and Applications},
	year      = {1993},
	volume    = {2},
	number    = {2},
	pages = {397--443}
}

@incollection{Maurer1981,
	author    = {Maurer, Helmut},
	title     = {First and second order sufficient optimality conditions in mathematical programming and optimal control},
	booktitle = {Mathematical Programming at Oberwolfach},
	year      = {1981},
	pages = {163--177},
	doi = {10.1007/bfb0120927},
	series = {Mathematical Programming Studies},
	volume = {14},
	publisher = {Springer},
	address = {Berlin, Heidelberg}
}

@article{McGill1965,
	author    = {McGill, Robert},
	title     = {Optimum control, inequality state constraints, and the generalized {Newton-Raphson} algorithm},
	journal   = {Journal of the Society for Industrial and Applied Mathematics, Series A: Control},
	year      = {1965},
	volume    = {3},
	number    = {2},
	pages = {291--298},
	doi = {10.1137/0303021}
}

@article{Mitter1966,
	author    = {Mitter, Sanjoy K.},
	title     = {Successive approximation methods for the solution of optimal control problems},
	journal   = {Automatica},
	year      = {1966},
	volume    = {3},
	number    = {3-4},
	pages = {135--149},
	doi = {10.1016/0005-1098(66)90009-4}
}

@incollection{Powell1978,
	author    = {Powell, Michael J. D.},
	title     = {The convergence of variable metric methods for nonlinearly constrained optimization calculations},
	booktitle = {Nonlinear programming 3},
	publisher = {Academic Press},
	year      = {1978},
	pages = {27--63},
	doi = {10.1016/b978-0-12-468660-1.50007-4}
}

@article{Powell1978MP,
	author    = {Powell, Michael J. D.},
	title     = {Algorithms for nonlinear constraints that use {Lagrangian} functions},
	journal   = {Mathematical Programming},
	year      = {1978},
	volume    = {14},
	number    = {1},
	pages = {224--248},
	doi = {10.1007/bf01588967}
}

@article{SchleyLee1967,
	author    = {Schley, Charles H. and Lee, Imsong},
	title     = {Optimal control computation by the {Newton-Raphson} method and the {Riccati} transformation},
	journal   = {IEEE Transactions on Automatic Control},
	year      = {1967},
	volume    = {12},
	number    = {2},
	pages = {139--144},
	doi = {10.1109/tac.1967.1098542}
}

@article{SunLiYong2016,
	author    = {Sun, Jingrui and Li, Xun and Yong, Jiongmin},
	title     = {Open-loop and closed-loop solvabilities for stochastic linear quadratic optimal control problems},
	journal   = {SIAM Journal on Control and Optimization},
	year      = {2016},
	volume    = {54},
	number    = {5},
	pages = {2274--2308},
	doi = {10.1137/15m103532x}
}

@article{Tapia1974,
	author    = {Tapia, Richard A.},
	title     = {{Newton's} method for problems with equality constraints},
	journal   = {SIAM Journal on Numerical Analysis},
	year      = {1974},
	volume    = {11},
	number    = {1},
	pages = {174--196},
	doi = {10.1137/0711018}
}

@incollection{Tapia1978,
	author    = {Tapia, Richard A.},
	title     = {Quasi-{Newton} methods for equality constrained optimization: Equivalence of existing methods and a new implementation},
	booktitle = {Nonlinear programming 3},
	publisher = {Academic Press},
	year      = {1978},
	pages = {125--164},
	doi = {10.1016/b978-0-12-468660-1.50010-4}
}

@article{Troltzsch1999,
	author    = {Tr{\"o}ltzsch, Fredi},
	title     = {On the {Lagrange-Newton-SQP} method for the optimal control of semilinear parabolic equations},
	journal   = {SIAM Journal on Control and Optimization},
	year      = {1999},
	volume    = {38},
	number    = {1},
	pages = {294--312},
	doi = {10.1137/s0363012998341423}
}

@book{Troltzsch2010,
	title={Optimal control of partial differential equations: theory, methods, and applications},
	author={Tr{\"o}ltzsch, Fredi},
	volume={112},
	year={2010},
	publisher = {American Mathematical Society},
	doi = {10.1090/gsm/112},
	series = {Graduate Studies in Mathematics},
	address = {Providence, RI}
}

@book{YongZhou1999,
	author    = {Yong, Jiongmin and Zhou, Xun Yu},
	title     = {Stochastic Controls: {Hamiltonian} Systems and {HJB} Equations},
	year      = {1999},
	publisher = {Springer},
	address   = {New York},
	doi = {10.1007/978-1-4612-1466-3}
}

\end{document}